\documentclass{article}
\usepackage{fullpage}
\usepackage{lipsum}
\usepackage{amsthm}
\usepackage{amsfonts}
\usepackage{graphicx}
\usepackage{epstopdf}
\usepackage{algorithm}
\usepackage{algorithmic}
\usepackage{booktabs}       
\usepackage{nicefrac}       
\usepackage{microtype}      
\usepackage{xcolor}         
\usepackage{tabularx}
\usepackage{mathtools}
\usepackage{pifont}
\newcommand{\cmark}{\ding{51}}
\newcommand{\xmark}{\ding{55}}
\allowdisplaybreaks[4]
\newtheorem{theorem}{Theorem}
\newtheorem{lemma}{Lemma}
\newtheorem{remark}{Remark}
\newtheorem{example}{Example}
\newtheorem{myassump}{Assumption}

\def\argmin{\mathop{\rm \boldmath argmin}}

\begin{document}
	
	\begin{center}
		{\bf {\LARGE{Robust stochastic first order methods in heavy-tailed noise via medoid mini-batch gradient sampling}}}
		
		\vspace*{.2in}
		
				\large{
					\begin{tabular}{cc}
						 Manojlo Vukovic$^\circ$, & Dusan Jakovetic$^\ast$,
					\end{tabular}
				}
			
		\vspace*{.2in}
			
		\begin{tabular}{c}
        $^\circ$University of Novi Sad, Faculty of Technical Sciences, Department of\\ Fundamental Sciences (manojlo.vukovic@uns.ac.rs)\\
		$^\ast$ University of Novi Sad, Faculty of Sciences, Department of\\ Mathematics and Informatics (dusan.jakovetic@dmi.uns.ac.rs)
		\\
		\end{tabular}
				
				\vspace*{.2in}
				
				\today
				
				\vspace*{.2in}
		
	\end{center}
	
	\begin{abstract}
	We consider a first order stochastic optimization framework where, at each iteration, $K$ independent identically distributed (i.i.d.) data point samples are drawn, based on which stochastic gradients can be queried. We allow gradient noise to be heavy-tailed, with possibly infinite variances. For the considered heavy-tailed setting, many algorithmic variants have recently been proposed based on gradient clipping or other nonlinear operators (e.g., normalization) applied over noisy gradients. In this paper, we take an alternative approach and propose a novel gradient estimator. The core idea of proposed estimator is to split the $K$-sized data batch into $M$ distinct data chunks, form for each chunk the stochastic gradient, and select the stochastic gradient of the chunk that is medoid of gradients of all data-chunks. For a broad class of symmetric heavy-tailed gradient noise distributions, we show that the proposed estimator has a finite second moment even when the underlying gradient noise has an infinite second moment. For nonconvex objectives, this yields convergence to a neighborhood of zero in expected squared gradient norm at a rate of $\mathcal{O}(T^{-1})$. We explicitly characterize the size of this neighborhood in terms of the noise and algorithm parameters. When the time horizon $T$ is known in advance, we establish an $\mathcal{O}(T^{-1/2})$ convergence rate. A diminishing step size achieves the same rate up to a logarithmic factor. Incorporating clipping or a more general nonlinear transformation yields high-probability convergence guarantees with the same rate up to logarithmic factors. Experimental results indicate that R-SGD-Mini and its clipped or nonlinear variant consistently perform favorably compared to SGD, clipped SGD, Nonlinear SGD and Median-of-Means based methods.
	\end{abstract}
	\section{Introduction}
First order stochastic optimization and stochastic gradient descent (SGD) have been widely used in many machine learning applications. 
    Initial SGD studies date back from~\cite{robinsmonro}, while the interest for the algorithm and its variants has been renewed thanks to many applications, including  widespread use in training large scale machine learning models~\cite{ghadimi,nemirovski}.\\
    Traditionally, SGD has been analyzed assuming that the gradient noise has a light-tailed distribution or at least has a finite variance.  However, in recent works~\cite{zhangHT,simsekliHT,nairHT,CSGDHT}, it has been shown that gradient noise may exhibit heavy tails when training certain deep learning models. 
%
 %
   This has motivated a surge of research recently on the development and analysis of SGD and its variants under heavy-tailed noise settings, e.g.,~\cite{CSGDHT,HTASGD,dusan,armacki,puchkinMedian}. Most works tackle the problem by applying a nonlinear mapping over the stochastic gradient direction before it is incorporated in the model update. For this, the clipping operator is dominantly used \cite{CSGDHT,HTASGD,puchkinMedian}, while other forms of more general nonlinearities such as normalization, sign, and quantization are also applied \cite{dusan,armacki}. 
  In these approaches, a core idea is that, by truncating excessively large values of a stochastic gradient, the effect of the noise is limited. 
  This also comes at the price of some information loss, as a portion of information regarding the noiseless gradient can also be lost. The studies then show that the tradeoff can be controlled carefully to achieve provable convergence and convergence rate guarantees in heavy-tailed settings.\\
\textbf{Technical novelty and contributions.}
    In this paper, we study the medoid\footnote{The medoid of sequence $\{\mathbf{a}_j\}_{j=1}^M$ of $M$ vectors is $\mathbf{a}_{j^\star}$ such that $j^\star=\argmin\limits_j\sum\limits_{i=1}^M \| \mathbf{a}_i-\mathbf{a}_j\|,$
    i.e., $\mathbf{a}_{j^\star}$ has the smallest distance to the other elements in the sequence.} of means estimator that, to the best of our knowledge, is used here for gradient estimation for the first time. For model updates, the resulting gradient estimator can be used either on its own or in combination with clipping, normalization or other nonlinearities.
    To illustrate the estimation procedure, we proceed as follows. First, we consider a stochastic optimization setting in which at each iteration~$t$, stochastic gradients can be evaluated based on $K$ independent, identically distributed (i.i.d.) data samples drawn from an unknown data distribution. Next, the $K$-sized data batch is then split into $M$ distinct data chucks, and for each  chunk, the stochastic gradient is formed; more precisely, they are set to equal the chunk-based averages of the per-data point and gradients. Finally, the gradient estimate is then chosen by using the gradient that corresponds to the chunk for which the stochastic gradient is the medoid  of all chunk gradients.
    A key insight behind medoid estimator is that, in a crude approximation (formal and detailed analysis is given in Sections~\ref{sec:assandalg} and~\ref{sec:preliminaries} ahead), the ``effective gradient noise'' that plays a role in the dynamics of the algorithm is the medoid noise of the $M$ noises that correspond to the different data chunks. In this context, a technical novelty is that the medoid noise has finite variance whenever $M$ is sufficiently large, even when the individual noises have infinite variance, leading to efficient operation of medoid estimator in heavy-tailed noise, as established by the results is Section~\ref{sec:preliminaries}.
    This contrasts with the standard average of the per-sample gradients, for which the "effective gradient noise" has infinite variance whenever the individual noise terms have infinite variance.
    
    Assuming throughout that the population (expected) loss is non-convex and has Lipschitz-continuous gradients, our analysis builds on the standard assumption used  to model heavy-tailed noise, namely we consider the bounded (central) $p$-th moment assumption with $p\in (1,2),$ which is
     $\mathbb{E}\left[\|\nabla f^t(\mathbf{x})-\nabla F(\mathbf{x})\|^p \right]\leq \sigma^p,$
where $\nabla F(\mathbf{x})$ and $\nabla f^t(\mathbf{x})$  are  the gradient of the objective function and stochastic gradient, respectively and $\sigma>0.$ Using the bounded $p$-th moment assumption, there has been a number of recent works in non-convex optimization, for example~\cite{CSGDHT,zhangHT,armacki,cutkosky2021high,liu2023breaking,nguyen2023improved}. On the other hand,~\cite{fatkhullin} study vanilla SGD without nonlinear gradient transformations and provide convergence guarantees assuming a bounded (p)-th moment of the stochastic gradient rather than a bounded $p$-th central moment, which is a more restrictive assumption. Moreover,~\cite{dusan,zhangHT}  show that SGD may fail to converge when the variance is infinite.
The theoretical results further rely on the symmetric gradient noise assumption. It is worth emphasizing that the state-dependent symmetric noise assumption is not unique to this work and has been previously considered in the prior works, including~\cite{chen2020understanding,bernstein2018signsgd,barsbey2021heavy,battash2024revisiting,armacki2025high,bernstein2018signsgd2}. 
Moreover,~\cite{barsbey2021heavy,gurbuzbalaban2021heavy,simsekliHT,peluchetti2020stable} theoretically show that symmetric heavy-tailed noise is a suitable noise model for many practical settings. 
This assumption can be relaxed by symmetrization when reference-point is available, as it is done in~\cite{armacki} (see also Remark~\ref{rem:assymetriextenison}). 
Moreover, allowing the sample size to grow with the iteration count makes it possible to drop the symmetry assumption and surpass the existing lower bounds (see also Remark~\ref{rem:increasingsamplesizes}). This approach, however, comes at a substantial computational cost.
It is worth noting that, without a symmetry assumption, a symmetrization strategy or with the increasing sample sizes, convergence guarantees have not yet been established for the proposed gradient estimator, whereas introducing a nonlinear transformation of the gradient can provide such guarantees under asymmetric noise. Establishing comparable guarantees without symmetrization remains an interesting direction for future research.

Under the above assumptions, we first highlight the following key strengths and contributions of our work. In particular, we emphasize that the proposed method (i) introduces a novel medoid-based chunk aggregation strategy, with the core contribution being a new gradient estimator together with a novel theoretical analysis, (ii) achieves robustness in the MSE sense solely through gradient selection, without relying on gradient clipping, normalization, truncation, or any other nonlinear tail-control mechanism, while matching the sharp state-of-the-art convergence rate of $O(T^{-1/2})$,
(iii) has optimal rate in the MSE sense regardless of the difficulty of the problem (for $M=5$), and 
(iv) can be readily incorporated into existing optimization frameworks (e.g.,~\cite{armacki} or~\cite{nguyen2023improved}), thereby inheriting their stronger high-probability convergence guarantees while simultaneously improving their robustness under heavy-tailed noise.

When the estimator is used on its own, we refer to the resulting method as Robust Stochastic Gradient Descent with medoid mini-batch gradient sampling, R-SGD-Mini for short, and derive bounds on the expected time-averaged squared gradient norm under a general class of symmetric heavy-tailed gradient noises, with a constant step size. To be more precise, we show that $\frac{1}{T}\sum\limits_{t=0}^T\mathbb{E}[\|\nabla F(\mathbf{x}^t)\|^2]$ converges at rate $T^{-1}$ to a neighborhood of zero. 
Therefore, if the time horizon $T$ is known in advance and we set step size $\alpha=a T^{-\frac{1}{2}},$ for $a>0$, we obtain the optimal rate, i.e., $\mathcal{O}(T^{-\frac{1}{2}}),$ which is the same as for SGD in the light-tailed setting. On the other hand, with a diminishing step size, R-SGD-Mini converges in the sense that the expected time-averaged squared gradient norm approaches zero at a rate of $\mathcal{O}(t^{-\frac{1}{2}}\log t).$
Alongside these results, we establish that medoid noise has finite second moment when $M$ is large enough, which is an important result in its own right and a key ingredient in proving the results above.
Furthermore, the fact that the effective noise has a finite second moment allows us to leverage existing results and obtain high-probability bounds from prior work that assumes finite second moment of gradient noise. Namely, by applying clipping to the medoid gradient estimator in the same manner as in~\cite{nguyen2023improved}\footnote{Here, the clipping threshold and the step size are chosen optimally to balance the bound.}, and using the fact that the noise in the medoid gradient estimator has a finite second moment, we recover the optimal convergence rate, i.e., $\mathcal{O}(T^{-\frac{1}{2}})$ with high probability. We refer to this modified algorithm as R-CSGD-Mini. Moreover, we also combine the medoid gradient estimator with the general class of nonlinearities considered in~\cite{armacki}. In this case, we refer to the resulting algorithm as R-NSGD-Mini.
Although the authors of~\cite{armacki} already establish the optimal convergence rate, we show that the upper bound on the corresponding constant is smaller when the nonlinearity is applied to medoid gradient estimator rather than to the standard averaged gradient (see Example~\ref{ex:vsarm}). 
Moreover, we show that using the proposed gradient estimator with a sample size that grows with the iteration index may recover the $\mathcal{O}(T^{-1})$ convergence bound of gradient descent. We refer to the resulting algorithm as R-iSGD-Mini. This approach, however, requires substantial computational resources (see Remark~\ref{rem:increasingsamplesizes} and Example~\ref{ex:increasingsamplesizes}).
    
\textbf{Literature review.}
A substantial body of literature considers settings that are identical or closely related to the one studied in this work. 
This includes several studies that provide relevant, convergence guarantees. For example, works such as~\cite{CSGDHT,zhangHT} achieve a rate of $\mathcal{O}(T^{-\frac{2p-2}{3p-2}})$ in expectation, while~\cite{cutkosky2021high,liu2023breaking,nguyen2023improved} establish the same rate under high-probability bounds. Consequently, under the assumption of symmetric gradient noise, our approach yields an improved rate whenever $p<2.$ 
Among the existing studies, several achieve the optimal convergence rate under their respective assumptions. We therefore compare our results with these studies in greater detail, focusing on their underlying assumptions and algorithmic mechanisms. 
For instance,~\cite{armackiocekivanje} obtain an $\mathcal{O}(t^{-1/2})$ convergence rate in MSE sense, up to logarithmic factors, in the nonconvex setting under the assumption of symmetric i.i.d. noise. In contrast, R-SGD-Mini achieves the same rate without requiring the noise to be independent or its density to be strictly positive around zero, and without introducing a nonlinear transformation into the algorithm. On the other hand, the analysis of~\cite{armackiocekivanje} does not require the (central) $p$-th moment  assumption imposed in our setting.
Furthermore,~\cite{armacki} establish an \(\mathcal{O}(T^{-1/2})\) convergence rate with high probability under the assumptions that the noise is symmetric and that its probability density function is strictly positive in a neighborhood of the origin. Our proposed R-CSGD-Mini attains the same rate when the gradient estimator is clipped in the same manner as in~\cite{nguyen2023improved}. They further derive high-probability guarantees for asymmetric noise when a reference point is available and the symmetrization procedure can therefore be applied. By contrast, when the noise is asymmetric and we employ the symmetrization procedure of~\cite{armacki}, our analysis provides a convergence guarantee only in the MSE sense. Furthermore, we show that applying the nonlinearity proposed by~\cite{armacki} to medoid gradient estimator yields a smaller constant in the convergence bound than applying it to the standard averaged gradient estimator (see Example~\ref{ex:vsarm}).
To the best of our knowledge, existing state-of-the-art methods do not explicitly address the optimization of the convergence constant, whereas we show that the proposed medoid estimator approach can significantly reduce it. Moreover, the resulting constant exhibits a more favorable dependence on both the problem dimension and the tail index.
Robust gradient estimators have also been considered by~\cite{puchkinMedian}, who establish a high-probability convergence rate of $\mathcal{O}(T^{-1/2})$ in the convex setting. More precisely, their method applies gradient clipping after finding element-wise gradient median of means. Moreover, in~\cite{puchkinMedian}, the authors do not rely on the bounded $p$-th moment assumption, nor on the assumption of symmetric noise. 
We further emphasize a key distinction from~\cite{puchkinMedian}: our gradient selection operates by choosing a single gradient from the set of chunk gradients, while~\cite{puchkinMedian} employs an element-wise median of gradients across chunks.
Although a natural extension of the element-wise median is the geometric median, which is coordinate-invariant, we do not pursue this approach as it requires solving an additional optimization problem whose complexity increases with the dimension. Instead, we approximate the geometric median by selecting the gradient that is the medoid of all chunk gradients and show that 
the medoid choice has similar properties as the geometric median, i.e., has a finite second moment even though the noise itself does not (see Example~\ref{ex:mo}).
We also compare the methods in terms of their computational cost. Let $F$ denote the cost of evaluating a single stochastic gradient.
For a single batch $K$, the computational cost is approximately
$T_{\mathrm{batch}} \approx O(K F).$ If the batch is divided into $M$ equally sized sub-batches, computing the gradient of one sub-batch costs approximately $O\left(\frac{K}{M}F\right).$ Since gradients are computed for all $M$ sub-batches, their total computational cost is
$M \cdot O\left(\frac{K}{M}F\right)=O(KF).$ 
Therefore, if the computational cost of processing the entire batch is $O(KF)$, using the element-wise median over $M$ mini-batch gradients results in a total cost of $O(KF+dM\log M)$, while using the exact medoid results in a total cost of $O(KF+M^2d)$. Thus, the effective computational overhead introduced by the medoid is $O(M^2d)$, which is indeed an additional cost compared with standard full-batch gradient computation.
However, our primary motivation was not to improve computational efficiency, but rather to capture the geometry of the problem while improving the tail behavior, without requiring the solution of the optimization problem associated with computing the geometric median.
Moreover, since $M$ is small in our setting and stochastic-gradient evaluations dominate the overall computational cost, this additional cost is expected to be relatively modest and may be negligible in practice.
On the other hand, the results presented in Section~\ref{sec:examples} show that the proposed medoid of means estimator has a shorter runtime than the element-wise median of means of~\cite{puchkinMedian}. Moreover, despite the computational overhead introduced by the medoid estimator, the numerical results demonstrate that the variants of R-SGD-Mini converge to a lower error plateau than the standard averaged-gradient estimator, irrespective of runtime.
These comparisons are summarized in Table~\ref{tab:algorithm_comparison}, which highlights the differences in the noise assumptions, objective-function class, and convergence guarantees.

\begin{table}
    \centering
    \setlength{\abovecaptionskip}{4pt}
    \renewcommand{\arraystretch}{1.15}

    \resizebox{\linewidth}{!}{%
    \begin{tabular}{@{}lcccccccc@{}}
        \toprule
        \textbf{Algorithm}
        & \textbf{No BM}
        & \textbf{No S}
        & \textbf{No UP}
        & \textbf{No IID}
        & \textbf{No ISS}
        & \textbf{F}
        & \textbf{C}
        & \textbf{R} \\
        \midrule
        R-SGD-Mini
        & \xmark & \xmark & \cmark & \cmark & \cmark
        & NC & E  & $\mathcal{O}(T^{-1/2})$ \\
        R-CSGD-Mini
        & \xmark & \xmark \footnotemark & \cmark & \cmark & \cmark
        & NC & HP & $\mathcal{O}(T^{-1/2})$ \\
        R-NSGD-Mini
        & \xmark & \xmark & \xmark & \cmark & \cmark
        & NC & HP & $\mathcal{O}(T^{-1/2})$ \\
        R-iSGD-Mini
        & \xmark & \cmark & \cmark & \cmark & \xmark
        & NC & E  & $\mathcal{O}(T^{-1})$ \\
        Nonlinear SGD~\cite{armackiocekivanje}
        & \cmark & \xmark & \xmark & \xmark & \cmark
        & NC & E  & $\mathcal{O}(T^{-1/2})$ \\
        NSGD~\cite{armacki}
        & \cmark & \xmark$^3$ & \xmark & \cmark & \cmark
        & NC & HP & $\mathcal{O}(T^{-1/2})$ \\
        Clipped MoM~\cite{puchkinMedian}
        & \cmark & \cmark & \cmark & \cmark & \cmark
        & C  & HP & $\mathcal{O}(T^{-1/2})$ \\
        \bottomrule
    \end{tabular}%
    }

    \caption{Comparison of convergence guarantees across algorithms.
    A check mark (\cmark) indicates that the stated guarantee holds without
    the corresponding condition, and a cross (\xmark) indicates that the
    guarantee requires it.
    BM denotes a bounded moment of order $p$, S symmetry, UP uniform
    positivity of the probability density, IID independent and identically
    distributed samples, and ISS increasing sample sizes.
    F denotes the objective class (NC for nonconvex and C for convex),
    C the mode of convergence (E for in expectation and HP for with high
    probability), and R the convergence rate.}
    \label{tab:algorithm_comparison}
\end{table} \footnotetext{The symmetry assumption can be dropped if a reference point is available.}
Mean estimators such as median, median of means and geometric median have been extensively studied in the problems of robust mean estimation and robust machine learning 
recently, e.g.~\cite{minsker2015geometric,lugosi2019mean,lugosi2019sub}. To the best of our knowledge, the medoid of means has not been used in optimization in the presence of heavy-tailed noise. On the other hand, $k$-medoids have been used in clustering~\cite{schubert2019faster,schubert2021fast}.\\	
\textbf{Paper organization}. The paper is structured as follows. Section~\ref{sec:assandalg} presents all assumptions and algorithm R-SGD-Mini. In Section~\ref{sec:preliminaries} some preliminary results are provided. Section~\ref{sec:main} presents the main results. In Section~\ref{sec:examples} we present numerical experiments. The conclusion is given in Section~\ref{sec:conlusion}.\\
 \textbf{Notation}. 
    We denote by ${\mathbb R}^d$ the $d$-dimensional
	Euclidean real coordinate space. We use normal lower-case letters for scalars,
	lower case boldface letters for vectors, and upper case boldface letters for
	matrices.
	We further denote by:
	$\|\cdot\|=\|\cdot\|_2$ the Euclidean norm of its vector argument; $\nabla h(\mathbf{w})$ and the gradient evaluated at $w$ of a function $h: {\mathbb R}^m \rightarrow {\mathbb R}$, $m > 1$; $ P(\mathcal A)$ and $\mathbb E[u]$ the probability of
	an event $\mathcal A$ and expectation of a random variable $u$, respectively.
    Notation $\mathcal{O}(\cdot)$ stands for the standard “big O”, i.e., for two non-negative sequences $\{a_t\}$ and $\{b_t\},$ we have $a_t = \mathcal{O}(b_t)$ if there exist $C >0$ and $t_0 \in\mathbb{N}$, such that $a_t \leq Cb_t$, for all $t\geq t_0.$

\section{Problem model and proposed algorithm}
	\label{sec:assandalg}
   We consider the following optimization problem:
    \begin{align}\label{eq:fullproblem}
    \argmin\limits_{\mathbf{x}\in\mathbb{R}^d}F(\mathbf{x})=\mathbb{E}_{\boldsymbol{\xi}} [f(\mathbf{x},\boldsymbol{\xi})].
    \end{align}
    Here, $\mathbf{x}\in\mathbb{R}^d$ corresponds to the model parameters to be learned, $f:\mathbb{R}^d\times\Xi\to \mathbb{R}$ is a per-data point loss function, $\boldsymbol{\xi}$ is a random variable drawn from (unknown) distribution $\Xi$, and $F:\mathbb{R}^d\to \mathbb{R}$ is the population loss. 
    We begin with the standard assumptions on the objective function  $F(\mathbf{x}).$
    \begin{myassump}\label{as:Lipshitzgradient} Function $F:\mathbb{R}^d\to \mathbb{R}$ is bounded from below and its gradient is $L$-Lipshitz, i.e.,  $\inf\limits_{\mathbf{x}\in\mathbb{R}^d} F(\mathbf{x})>-\infty$ and for all $\mathbf{x},\mathbf{y}\in\mathbb{R}^d$ we have that $
            \|\nabla F(\mathbf{x})-\nabla F(\mathbf{y})\|\leq L \|\mathbf{x}-\mathbf{y}\|.$\footnote{We denote $F^\star := \inf\limits_{\mathbf{x}\in\mathbb{R}^d} F(\mathbf{x})$. Moreover, $L$-Lipschitz continuity of $\nabla F$ gives us the following quadratic upper bound: for all $\mathbf{x},\mathbf{y}\in\mathbb{R}^d,$ $ F(\mathbf{y})\leq F(\mathbf{x})+\nabla F^\top(\mathbf{x}) (\mathbf{y}-\mathbf{x})+\frac{L}{2}\|\mathbf{y}-\mathbf{x}\|^2.$ }
    \end{myassump}
    To solve problem~\eqref{eq:fullproblem}, the standard SGD method can be applied, where instead of using the true gradient $\nabla F(\mathbf{x})$ of population loss $F(\mathbf{x}),$ one uses its estimate $\nabla f(\mathbf{x}).$ That is, there is noise $\boldsymbol{\nu}(\mathbf{x})$ that arises from estimating the true gradient using a finite sample, i.e., we have that the gradient estimate $\nabla f(\mathbf{x})$ can be written as $\nabla f(\mathbf{x})=\nabla F(\mathbf{x})  + \boldsymbol{\nu}(\mathbf{x}).$ Gradient noise, depending on the sampling method and the true gradient estimator, can have various properties. It is usually assumed that the noise $\boldsymbol{\nu}(\mathbf{x})$ has zero mean and that its variance is finite or at least bounded by some function of $\mathbf{x}.$ However, due to the development of deep neural networks and the use of very large datasets, the assumption of bounded variance is no longer applicable. That is, heavy-tailed noise is increasingly considered, which in fact is the class of noises that may have infinite variance. In such a setting, it can be shown that standard SGD may diverge (see, for example, Appendix B in~\cite{dusan}). Existing algorithms typically rely on nonlinear transformations, such as clipping, normalization, or other general nonlinearities, to ensure that the effective noise has finite variance. In contrast, we propose a gradient estimator that has finite variance, while still allowing such nonlinear transformations to be applied if desired.
    
    To do that, in this paper, we consider the following strategy: at each iteration sample $K$ data samples and split the data batch into $M=\frac{K}{R}$ distinct data chunks and form for each chunk $j=1,2,...,M$ the corresponding stochastic gradient $\nabla f_j^t(\cdot)=\frac{1}{R}\sum\limits_{k=1}^R \nabla f(\cdot,\boldsymbol{\xi}_{j_k}),$ and choose $j^\star$ for which the corresponding noise exhibits the best characteristics with respect to the predefined metric. Before we proceed, let us introduce the notation for gradient noise for each of the data chunk $j=1,2,..,M:$
    \begin{align}
        \nabla f_j^t(\mathbf{x}) &=\nabla F(\mathbf{x})  + \boldsymbol{\nu}_j^t(\mathbf{x}).\label{eqn:gradnoise}
    \end{align}
    Thus, we need to design a strategy that selects a stochastic gradient $\nabla f_j^t$ whose corresponding noise term $\boldsymbol{\nu}_j^t(\mathbf{x})$ exhibits the most favorable characteristics. For example, one can choose $j^\star$ such that 
    \begin{align}\label{eqn:badjstar}
        j^\star= \argmin_j \|\nabla f_j^t(\mathbf{x}^t) \|= \argmin_j \|\nabla F(\mathbf{x})  + \boldsymbol{\nu}_j^t(\mathbf{x})\|.
    \end{align}
    However, even though it would provide that effective variance is finite, it would make that mean of $\boldsymbol{\nu}_{j^\star}^t(\mathbf{x})$ is not zero nor close to zero 
    \begin{remark} \label{rem:badjstar}
        Notice that if we had i.i.d. zero-mean symmetric random vectors $\boldsymbol{\nu}_1,\boldsymbol{\nu}_2,...,\boldsymbol{\nu}_M$ and $j^\star$ is chosen such that
        \begin{align}\label{eqn:badjstar2}
             j^\star= \argmin_j \|\boldsymbol{\nu}_j\|,
        \end{align}
        we would have $\mathbb{E}[\boldsymbol{\nu}_{j^\star}]=0.$ However, we have that $j^\star$ in~\eqref{eqn:badjstar} is not the same as~\eqref{eqn:badjstar2}, i.e., even tough we have
            $\min_j \|\nabla f_j^t(\mathbf{x}^t) \| = \min_j \|\nabla F(\mathbf{x})  + \boldsymbol{\nu}_j^t(\mathbf{x})\| \leq \|\nabla F(\mathbf{x})\| + \min_j \|\boldsymbol{\nu}_j^t(\mathbf{x})\|,$
        we do not have
            $\argmin_j \|\nabla f_j^t(\mathbf{x}^t) \| =\argmin_j  (\|\nabla F(\mathbf{x})\| + \|\boldsymbol{\nu}_j^t(\mathbf{x})\|).$
    \end{remark}
    Thus, in the view of Remark~\ref{rem:badjstar}, the main difficulty in finding $j^\star$ such that effective noise retains nice properties is the true gradient term $\nabla F(\mathbf{x}).$ This means that we have to neglect the influence of $\nabla F(\mathbf{x})$ in the process of finding $j^\star.$ One possible way to achieve this is to find $j^\star$ in the following manner
    \begin{align*}
         j^\star= \argmin_j \sum\limits_{i=1}^N \mathbf{h}( \nabla f_i^t(\mathbf{x}) - \nabla f_j^t(\mathbf{x}))= \argmin_j \sum\limits_{i=1}^N \mathbf{h}(\boldsymbol{\nu}_i^t(\mathbf{x})- \boldsymbol{\nu}_j^t(\mathbf{x})),
    \end{align*}
    where $\mathbf{h}:\mathbb{R}^d\to\mathbb{R}$ is some function.
    \begin{remark}\label{rem:choicePsi}
    Let $\boldsymbol{\nu}_1, \boldsymbol{\nu}_2,..., \boldsymbol{\nu}_M$ be i.i.d. random variables. Then, depending on the choice of the function $\mathbf{h}$ in 
        $\mathbf{y}=\argmin_{\mathbf{y}\in\mathbb{R}^d} \sum\limits_{i=1}^N \mathbf{h}( \boldsymbol{\nu}_i - \mathbf{y})$
    we obtain different estimators of the mean. For example, 1) if $\mathbf{h}(\cdot)=\|\cdot\|_1,$ we get the element-wise median;
     2) if $\mathbf{h}(\cdot)=\|\cdot\|^2,$ we get the mean;
      3) if $\mathbf{h}(\cdot)=\|\cdot\|,$ we get the geometric median;  
    \end{remark}
     We deal with the case when the function $\mathbf{h}$ is chosen such that $\mathbf{h}(\cdot)=\|\cdot\|,$ i.e.,
    \begin{align*}
         j^\star= \argmin_j \sum\limits_{i=1}^N \| \nabla f_i^t(\mathbf{x}) - \nabla f_j^t(\mathbf{x})\| = \argmin_j \sum\limits_{i=1}^N \| \boldsymbol{\nu}_i^t(\mathbf{x})- \boldsymbol{\nu}_j^t(\mathbf{x})\|.
    \end{align*}
    That means we choose the data chunk whose gradient is medoid of $\nabla f_i^t(\mathbf{x})$ for $i=1,2,...,M$, that is the gradient $\nabla f_{j^\star}^t(\mathbf{x})$ that is closest to all the other gradients $\nabla f_i^t(\mathbf{x})$ with respect to $\|\cdot\|$. Even though, it is not true geometric median, it exhibits similar properties as geometric median itself, as it can be seen in Section~\ref{sec:preliminaries}. Moreover, notice that finding $j^\star$ is computationally cheaper than finding the true geometric median.

    We now present the assumptions imposed on the gradient noise.

    \begin{myassump}\label{as:gradientnoise} \textbf{Gradient noise:}
        1) For every fixed $\mathbf{x},$ the per-point gradient noises $\boldsymbol{\nu}_{\boldsymbol{\xi}}(\mathbf{x})=\nabla f(\mathbf{x},\boldsymbol{\xi})-\nabla F(\mathbf{x})$ are absolutely continuous and identically distributed;
        2) For every fixed $\mathbf{x},$ $\boldsymbol{\nu}_{\boldsymbol{\xi}}(\mathbf{x})$ and $-\boldsymbol{\nu}_{\boldsymbol{\xi}}(\mathbf{x})$ have the same distribution
        3) For every fixed $\mathbf{x}$ and every $K,$ $\{\boldsymbol{\nu}_{\boldsymbol{\xi}_i}(\mathbf{x})\}_{i=1}^K$ are mutually independent.
        4) There exist constants $\sigma>0$ and $p>1$ such that
            $\sup\limits_{\mathbf{x}}\mathbb{E}\left[ \|\boldsymbol{\nu}_{\boldsymbol{\xi}}(\mathbf{x})\|^p \right] \leq \sigma^p$
    \end{myassump}
    Assumption~\ref{as:gradientnoise} allows us to consider a wide range of gradient noises. We require the per-sample gradient noise to have a finite moment of order $p>1$ and to follow an absolutely continuous, symmetric distribution. 
   Since $p$ can be as small as a number arbitrarily close to $1,$  we actually allow for very heavy noise tails. In particular, the gradient noise variance may be equal to infinity. 
   The symmetry assumption can be relaxed, as explained in Remark~\ref{rem:assymetriextenison}. Regarding absolute continuity, machine-learning optimization typically considers either finite-sum optimization problems or stochastic optimization problems formulated as the minimization of an expectation, where the expectation is taken with respect to a random variable following a given probability distribution. In this work, we consider the latter setting, in which there is no obstacle to assuming that the gradient noise is absolutely continuous. Moreover, the assumption that the gradient noise admits an absolutely continuous probability density function is standard in the literature and has been adopted in numerous previous works (see, e.g., ~\cite{dusan,puchkinMedian}).

    Based on this assumption, we have that the moment of order $p$ of gradient noise for each data chunk $j=1,2,...,M$ is given by
   $\sup\limits_{\mathbf{x}}\mathbb{E}\left[ \|\boldsymbol{\nu}_j^t(\mathbf{x})\|^p \right] \leq \frac{1}{R^{p-1}} \sup\limits_{\mathbf{x}}\mathbb{E}\left[ \|\boldsymbol{\nu}_{\boldsymbol{\xi}}(\mathbf{x})\|^p \right] \leq \frac{\sigma^p}{R^{p-1}}.$

   Utilizing Markov inequality, we have that probability tails for gradient noise $\boldsymbol{\nu}_j^t(\mathbf{x})$ for all $j=1,2,...,M$ are given by
   \begin{align}\label{eqn:datachunktail}
       \sup\limits_{\mathbf{x}} P\{ \|\boldsymbol{\nu}_j^t(\mathbf{x})\|>u\} \leq \frac{\sigma^p}{R^{p-1}u^p}.
   \end{align}
    We now present the pseudocode of the R-SGD-Mini.
    \begin{algorithm}[H]
    \caption{R-SGD-Mini (\underline{R}obust \underline{S}tochastic \underline{G}radient Descent with medoid \underline{mini}-batch gradient sampling}\label{alg1}
    \textbf{Require}: initial value $\mathbf{x}^0,$ learning rate $\alpha,$ sample size $K,$ the number of data points in each data chunk $R$ ($M=\frac{K}{R}$)
    \begin{algorithmic}[1]
    \STATE Set $t=0$
    \WHILE{exit criteria is not satisfied}
    \STATE Sample $K$ data samples $\boldsymbol{\xi}_1, \boldsymbol{\xi}_2,...,\boldsymbol{\xi}_K$ from distribution $\Xi$, and split them into $M=\frac{K}{R}$ equally-sized data chunks 
    \STATE Set that $\nabla f_j^t(\cdot)=\frac{1}{R}\sum\limits_{k=1}^R \nabla f(\cdot,\boldsymbol{\xi}_{j_k})$ for $j=1,2,...,M$
    \STATE Find index $j^\star$ such that
       $j^\star= \argmin_j \sum\limits_{i=1}^N \| \nabla f_i^t(\mathbf{x}) - \nabla f_j^t(\mathbf{x})\|$ and set $\mathbf{g}^t=\nabla f_{j^\star}^t(\mathbf{x}^t)$
    \STATE Set $\mathbf{x}^{t+1}=\mathbf{x}^t-\alpha \mathbf{g}^t,$
    \STATE Set $t=t+1.$
    \ENDWHILE
    \end{algorithmic}
    \end{algorithm}
    Note that, in Step 5, where $\mathbf{g}^t$ is computed, applying clipping as in~\cite{nguyen2023improved} or  general nonlinearity form~\cite{armacki} to the medoid gradient estimator $\nabla f_{j^\star}^t(\mathbf{x}^t)$, instead of using the medoid gradient directly, yields R-CSGD-Mini and R-NSGD-Mini, respectively.\\
    To conclude this section, we note that, throughout the paper, we consider the filtration $\mathcal{F}_t$, $t= 1,2,...,$ where $\mathcal{F}_t$ is the $\sigma$-algebra generated by $\{ \boldsymbol{\nu}_j^s\}_{s=0}^{t-1}$ for $j=1,2,...,M.$

\section{Intermediate results}\label{sec:preliminaries}
This section provides some intermediate technical results that are required to later carry our convergence analysis of Algorithm~\ref{alg1}. For simplicity of presentation and notation, throughout this section, we drop the dependencies of gradient noise $\boldsymbol{\nu}_j^t(\mathbf{x})$ on $\mathbf{x}$  and $t$. Dropping the iteration $t$ and  $\mathbf{x}$ is justified since we analyze the behavior at a fixed iteration. 
    
Let us denote by $D(\boldsymbol{\nu})$ the sum of distances between point $\boldsymbol{\nu}$ and points $\boldsymbol{\nu}_i,i=1,2,...,M,$ i.e., $D(\boldsymbol{\nu})=\sum_{i=1}^M\|\boldsymbol{\nu}_i-\boldsymbol{\nu}\|.$ Let $\boldsymbol{\nu}_{j^\star}$ be one among all $i=1,2,...,M$ that minimizes $D(\boldsymbol{\nu}),$ i.e.  $j^\star=\argmin_j D(\boldsymbol{\nu}_j).$ As already mentioned, that $\boldsymbol{\nu}_{j^\star}$ is called medoid.
We start with the following lemma, which establishes that if medoid $\boldsymbol{\nu}_{j^\star}$ is far from a given point $\mathbf{z},$ then a fraction of the points $\boldsymbol{\nu}_i,i=1,2,...,M,$ is also far from that point $\mathbf{z}.$
Unlike~\cite{minsker2015geometric} for a fraction of the points to also be sufficiently distant from $\mathbf{z}$, the distance of the medoid from point $\mathbf{z}$ also depends on the total number of points considered.
 Moreover, in the following lemma we additionally assume that the function $D(\boldsymbol{\nu})$ does not have the same value for any two $\boldsymbol{\nu}_j,\boldsymbol{\nu}_k $ for $i\neq k$, which was not required for geometric median result.
\begin{lemma}\label{lemma:farenough}
    Let $\boldsymbol{\nu}_1,\boldsymbol{\nu}_2,...,\boldsymbol{\nu}_M \in\mathbb{R}^d$ such that $D(\boldsymbol{\nu}_j)\neq D(\boldsymbol{\nu}_k)$ for all $i\neq k$ and let $\boldsymbol{\nu}_{j^\star}$ be their medoid. Fix $\gamma\in(0,\frac{1}{2})$ and $M\in\mathbb{N}$ and assume that $\mathbf{z}\in\mathbb{R}^d$ is such that $\|\boldsymbol{\nu}_{j^\star}-\mathbf{z}\| > C_{\gamma,M} r,$
    where
        $C_{\gamma,M}=\frac{(3-2\gamma)M+3}{(1-2\gamma)M}$
    and $r>0.$ Then there exists a subset $J\subseteq \{1,2,...,M\}$ of cardinality $|J|>\gamma M$ such that for all $j\in J,$ $\|\boldsymbol{\nu}_j-\mathbf{z}\|>r.$
\end{lemma}
\begin{proof}
    Suppose that the implication does not hold. Without loss of generality, this implies that $\|\boldsymbol{\nu}\|\leq r,$ for $i=1,2,...,\lfloor(1-\gamma)M\rfloor+1.$ If we consider some $\boldsymbol{\nu}_k$ such that $k\in \{1,2,...,\lfloor(1-\gamma)M\rfloor+1\},$ and show that $D(\boldsymbol{\nu}_{j^\star})-D(\boldsymbol{\nu}_k)>0,$ we reach a contradiction, which establishes the lemma. We have
    \begin{align*}
        &D(\boldsymbol{\nu}_{j^\star})-D(\boldsymbol{\nu}_k) = \sum_{i=1}^M \left(\|\boldsymbol{\nu}_i-\boldsymbol{\nu}_{j^\star}\|- \|\boldsymbol{\nu}_i-\boldsymbol{\nu}_k\| \right)\\
        & = \sum_{i=1}^{\lfloor(1-\gamma)M\rfloor+1} \left(\|\boldsymbol{\nu}_i-\boldsymbol{\nu}_{j^\star}\|- \|\boldsymbol{\nu}_i-\boldsymbol{\nu}_k\| \right)+
        \sum_{i=\lfloor(1-\gamma)M\rfloor+2}^{M} \left(\|\boldsymbol{\nu}_i-\boldsymbol{\nu}_{j^\star}\|- \|\boldsymbol{\nu}_i-\boldsymbol{\nu}_k\| \right)
    \end{align*}
    For $i=1,2,...,\lfloor(1-\gamma)M\rfloor+1,$ we have that
    \begin{align*}
        \|\boldsymbol{\nu}_i-\boldsymbol{\nu}_{j^\star}\|&=\|\boldsymbol{\nu}_i-\mathbf{z}-(\boldsymbol{\nu}_{j^\star}-\mathbf{z})\| \geq \|\boldsymbol{\nu}_{j^\star}-\mathbf{z}\| - \|\boldsymbol{\nu}_i-\mathbf{z}\| \geq \|\boldsymbol{\nu}_{j^\star}-\mathbf{z}\| - r\\
       \|\boldsymbol{\nu}_i-\boldsymbol{\nu}_k\|&=\|\boldsymbol{\nu}_i-\mathbf{z}-(\boldsymbol{\nu}_k-\mathbf{z})\|\leq \|\boldsymbol{\nu}_i-\mathbf{z}\|+\|\boldsymbol{\nu}_k-\mathbf{z}\| \leq2r,
    \end{align*}
    hence, $\|\boldsymbol{\nu}_i-\boldsymbol{\nu}_{j^\star}\|- \|\boldsymbol{\nu}_i-\boldsymbol{\nu}_k\|\geq \|\boldsymbol{\nu}_{j^\star}-\mathbf{z}\| -3r,$ and
    \begin{align*}
        \sum_{i=1}^{\lfloor(1-\gamma)M\rfloor+1} \left(\|\boldsymbol{\nu}_i-\boldsymbol{\nu}_{j^\star}\|- \|\boldsymbol{\nu}_i-\boldsymbol{\nu}_k\| \right) \geq \left(\lfloor(1-\gamma)M\rfloor+1\right) \left( \|\boldsymbol{\nu}_{j^\star}-\mathbf{z}\| -3r\right).
    \end{align*}
    On the other hand, for $i=\lfloor(1-\gamma)M\rfloor+2, \lfloor(1-\gamma)M\rfloor+3,...,M$ we have that
    \begin{align*}
        \|\boldsymbol{\nu}_i-\boldsymbol{\nu}_{j^\star}\|- \|\boldsymbol{\nu}_i-\boldsymbol{\nu}_k\| &\geq - \| \boldsymbol{\nu}_{j^\star}- \boldsymbol{\nu}_k\|= - \| \boldsymbol{\nu}_{j^\star}-\mathbf{z} - (\boldsymbol{\nu}_k-\mathbf{z})\|\\ 
        &\geq - (\| \boldsymbol{\nu}_{j^\star}-\mathbf{z}\|+ \| \boldsymbol{\nu}_k-\mathbf{z}\|) \geq -\left(\|\boldsymbol{\nu}_{j^\star}-\mathbf{z}\| + r\right),
    \end{align*}
    and
        $\sum_{i=\lfloor(1-\gamma)M\rfloor+2}^{M} \left(\|\boldsymbol{\nu}_i-\boldsymbol{\nu}_{j^\star}\|- \|\boldsymbol{\nu}_i-\boldsymbol{\nu}_k\| \right) \geq - (\lceil \gamma M\rceil -1) \left(\|\boldsymbol{\nu}_{j^\star}-\mathbf{z}\| + r \right).$
    We have that
    \begin{align*}
        D(\boldsymbol{\nu}_{j^\star})&-D(\boldsymbol{\nu}_k) \geq \left(\lfloor(1-\gamma)M\rfloor+1\right) \left( \|\boldsymbol{\nu}_{j^\star}-\mathbf{z}\| -3r\right) - (\lceil \gamma M\rceil -1) \left(\|\boldsymbol{\nu}_{j^\star}-\mathbf{z}\| + r \right)\\
        &= \left(\lfloor(1-\gamma)M\rfloor - \lceil \gamma M\rceil +2 \right) \|\boldsymbol{\nu}_{j^\star}-\mathbf{z}\| - \left( 3\lfloor(1-\gamma)M\rfloor + \lceil \gamma M\rceil +2\right)r
    \end{align*}
    Using $x-1<\lfloor x \rfloor \leq x$ and $x\leq \lceil x\rceil <x+1,$ we get
    \begin{align*}
         D(\boldsymbol{\nu}_{j^\star})-D(\boldsymbol{\nu}_k)&\geq (1-2\gamma)M \|\boldsymbol{\nu}_{j^\star}-\mathbf{z}\| - ((3-2\gamma)M+3) r>0
    \end{align*}
    if $ \|\boldsymbol{\nu}_{j^\star}-\mathbf{z}\| > \frac{(3-2\gamma)M+3}{(1-2\gamma)M}r,$ which leads to a contradiction.
\end{proof}

The following lemma guaranties that the condition $D(\boldsymbol{\nu}_j)= D(\boldsymbol{\nu}_k)$ for all $i\neq k$ we imposed for arbitrary points $\boldsymbol{\nu}_i,i=1,2,...,M,$ in the $\mathbb{R}^d$ occurs with probability zero when they are absolutely continuous random variables.

\begin{lemma}\label{lemma:absolutecontinuous}
    Let $\boldsymbol{\nu}_i,i=1,2,...,M$ be i.i.d. absolute continuous random variables. We have that $P\{D(\boldsymbol{\nu}_i) =D(\boldsymbol{\nu}_j)\} =0 $ for all $i\neq j.$
\end{lemma}
\begin{proof}
Fix $\boldsymbol{\nu}_i,i=2,3,...,M$ and let $\boldsymbol{\nu}_1=\boldsymbol{\nu}$ for some $\boldsymbol{\nu}\in\mathbb{R}^d.$ Denote by $G(\boldsymbol{\nu})=D(\boldsymbol{\nu}) - D(\boldsymbol{\nu}_2).$ We have that
\begin{align*}
    G(\boldsymbol{\nu})&=\sum_{i=1}^M\|\boldsymbol{\nu}_i-\boldsymbol{\nu}\| - \sum_{i=1}^M\|\boldsymbol{\nu}_i-\boldsymbol{\nu}_2\|
    =\sum_{i=3}^M\|\boldsymbol{\nu}_i-\boldsymbol{\nu}\| - \sum_{i=3}^M\|\boldsymbol{\nu}_i-\boldsymbol{\nu}_2\|,
\end{align*}
and second sum is constant with respect to $\boldsymbol{\nu}.$ Since $P\{\boldsymbol{\nu}= \boldsymbol{\nu}_i\}=0$ for $i=3,4,...,M$ it is enough to consider the gradient $\nabla G(\boldsymbol{\nu})$
 for $\boldsymbol{\nu}\neq \boldsymbol{\nu}_i, i=3,4,...,M,$ that is
$\nabla G(\boldsymbol{\nu})=\sum_{i=3}^M\frac{\boldsymbol{\nu}_i-\boldsymbol{\nu}}{\|\boldsymbol{\nu}_i-\boldsymbol{\nu}\|}.$
Hence, $\nabla G(\boldsymbol{\nu})=0$ only if $\boldsymbol{\nu}$ is a true geometric median of $\boldsymbol{\nu}_i,i=3,4,...,M.$ Therefore, set $\mathcal{A}=\{\boldsymbol{\nu} | \sum_{i=3}^M\frac{\boldsymbol{\nu}_i-\boldsymbol{\nu}}{\|\boldsymbol{\nu}_i-\boldsymbol{\nu}\|}=0 \}$ has Lebesgue measure zero in $\mathbb{R}^d.$ This concludes the proof since then we have that $P\{\boldsymbol{\nu}\in \mathcal{A}\}=0.$
\end{proof}

We are now ready to derive a result analogous to the result for geometric median, as in~\cite{minsker2015geometric}.
\begin{theorem}\label{theo:probabilitybound} Let $\boldsymbol{\nu}_i,i=1,2,...,M$ be i.i.d. absolute continuous random variables with mean $\boldsymbol{\mu}.$ Fix $\gamma\in(0,\frac{1}{2})$ and $M\in\mathbb{N}.$ Let $0<q<\gamma$ and $\varepsilon>0$ be such that for all $j=1,2,...,M$ $P\{ \| \boldsymbol{\nu}_j-\boldsymbol{\mu}\| > \varepsilon\}\leq q.$
If $\boldsymbol{\nu}_{j^\star}$ is their medoid, i.e., then
$P\{ \| \boldsymbol{\nu}_{j^\star}-\boldsymbol{\mu}\| > C_{\gamma,M}\varepsilon\}\leq \exp{\left(-M\psi(\gamma,q) \right)},$
where $C_{\gamma,M}$ is from Lemma~\ref{lemma:farenough} and 
$\psi(\gamma,q) = (1-\gamma)\log \frac{1-\gamma}{1-q}+\gamma \log\frac{\gamma}{q}.$
\end{theorem}
\begin{proof}
    Let us consider the event $\mathcal{E}=\{D(\boldsymbol{\nu}_i) \neq D(\boldsymbol{\nu}_j) \text{ for all }i\neq j\}.$
    We have that
    \begin{align*}
        P\{ \| \boldsymbol{\nu}_{j^\star}-\boldsymbol{\mu}\| > C_{\gamma,M}\varepsilon\} &= P\{ \{\| \boldsymbol{\nu}_{j^\star}-\boldsymbol{\mu}\| > C_{\gamma,M}\varepsilon\} \cap \mathcal{E}  \}\\&+P\{ \{ \| \boldsymbol{\nu}_{j^\star}-\boldsymbol{\mu}\| > C_{\gamma,M}\varepsilon\}\cap \mathcal{E}^c  \}
    \end{align*}
    Imposing Lemmas~\ref{lemma:farenough} and~\ref{lemma:absolutecontinuous} we have that
    \begin{align*}
        P\{ \| \boldsymbol{\nu}_{j^\star}-\boldsymbol{\mu}\| > C_{\gamma,M}\varepsilon\} &\leq P\{\sum_{i=1}^M I(\|\boldsymbol{\nu}_i\|>\varepsilon )>\gamma M \}\leq P\{ \xi >\gamma M\},
    \end{align*}
    where $\xi$ has Binomial distribution $\xi\sim B(M,q).$ Utilizing Chernoff bound we have that
    \begin{align*}
        P\{ \| \boldsymbol{\nu}_{j^\star}-\boldsymbol{\mu}\| > C_{\gamma,M}\varepsilon\} \leq \exp{\left(-M\psi(\gamma,q) \right)}.
    \end{align*}
\end{proof}

Now we will apply Theorem~\ref{theo:probabilitybound} to the gradient noises of the data chunks $\boldsymbol{\nu}_j$ for all $j=1,2,...,M.$ From~\eqref{eqn:datachunktail} we have that $P\{ \|\boldsymbol{\nu}_j^t\|>u\} \leq \frac{\sigma^p}{R^{p-1}u^p}$ and 
\begin{align}\label{eqn:p}
    P\{ \|\boldsymbol{\nu}_j^t\|>\frac{u}{C_{\gamma,M}}\} \leq \frac{\sigma^pC_{\gamma,M}^p}{R^{p-1}u^p}= \frac{\sigma^p((3-2\gamma)M+3)^p}{R^{p-1} ((1-2\gamma)M)^p u^p}=:q.
\end{align}
For $u>\frac{\sigma((3-2\gamma)M+3)}{R^{\frac{p-1}{p}}((1-2\gamma)M)\gamma^{\frac{1}{p}} }$ we have $q<\gamma.$ Thus,
    $P\{ \| \boldsymbol{\nu}_{j^\star}\| > u\}\leq \exp{\left(-M\psi(\gamma,q) \right)},$
and $\psi(\gamma,q) = (1-\gamma)\log \frac{1-\gamma}{1-q}+\gamma \log\frac{\gamma}{q}
    =\log \left(\frac{1-\gamma}{1-q}\right)^{1-\gamma}\left(\frac{\gamma}{q}\right)^\gamma,$
 hence,$-M\psi(\gamma,q) = \log \left( \left(\frac{1-q}{1-\gamma}\right)^{1-\gamma}\left(\frac{q}{\gamma}\right)^\gamma\right)^M.$
Now we consider the function $h(q,\gamma)=\left(\frac{1-q}{1-\gamma}\right)^{1-\gamma}\left(\frac{q}{\gamma}\right)^\gamma.$ Fix $\theta\in(0,1),$ we have
$h(q,\gamma)=\left(\frac{1-q}{1-\gamma}\right)^{1-\gamma} \frac{q^{(1-\theta)\gamma}}{\gamma^\gamma} q^{\theta \gamma}$
and $\left(\frac{1-q}{1-\gamma}\right)^{1-\gamma} \frac{q^{(1-\theta)\gamma}}{\gamma^\gamma}<1$ if $q<\left((1-\gamma)^{1-\gamma}\gamma^\gamma\right)^{\frac{1}{(1-\theta)\gamma}}.$ Notice that $\left((1-\gamma)^{1-\gamma}\gamma^\gamma\right)^{\frac{1}{(1-\theta)\gamma}}<\gamma.$ Therefore, for $u>q_1,$ for 
\begin{align}\label{eqn:q}
q_1 :=\frac{\sigma((3-2\gamma)M+3)}{R^{\frac{p-1}{p}}((1-2\gamma)M) \left((1-\gamma)^{1-\gamma}\gamma^\gamma\right)^{\frac{1}{(1-\theta)p\gamma}} }   
\end{align}
we get $\exp{\left(-M\psi(\gamma ,p) \right)} \leq \left(\frac{\sigma^p((3-2\gamma)M+3)^p}{R^{p-1} ((1-2\gamma)M)^p u^p}\right)^{\theta\gamma M}=\left(\frac{\sigma^p((3-2\gamma)M+3)^p}{R^{p-1} ((1-2\gamma)M)^p }\right)^{\theta\gamma M}\frac{1}{u^{p\theta\gamma M}}$

Finally, we are ready to prove the following theorem, which provides bound on the second moment.

\begin{theorem}\label{theo:variance}
     Suppose that $\boldsymbol{\nu}_j,$ for $j=1,2,...,M$ are random variables that satisfy Assumption~\ref{as:gradientnoise} and suppose that $\boldsymbol{\nu}_{j^\star}$ is their medoid. Fix $\gamma\in(0,\frac{1}{2}),$ $\theta\in(0,1)$ and let $M>\frac{2}{p\theta\gamma}.$ Then we have that 
     \begin{align}
         \mathbb{E}\left[\|\boldsymbol{\nu}_{j^\star}\|^2 \right] &\leq \left(\frac{\sigma((3-2\gamma)M+3)}{R^{\frac{p-1}{p}}((1-2\gamma)M) \left((1-\gamma)^{1-\gamma}\gamma^\gamma\right)^{\frac{1}{(1-\theta)p\gamma}} }\right)^2 \left(1+\frac{2}{p\theta\gamma M -2} \right).\label{eqn:truesecondmoment}
     \end{align}
\end{theorem}
\begin{proof}
    For $p\theta\gamma M >2,$ and recall $q_1$ from~\eqref{eqn:q} we get
    \begin{align*}
        &\mathbb{E}\left[\|\boldsymbol{\nu}_{j^\star}\|^2 \right]\\&=2\int_0^\infty uP\{ \|\boldsymbol{\nu}_{j^\star}\| >u\} du= 2\int_0^q u P\{ \|\boldsymbol{\nu}_{j^\star}\| >u\} du + 2\int_q^\infty u P\{ \|\boldsymbol{\nu}_{j^\star}\| >u\} du\\
        & \leq \left(\frac{\sigma((3-2\gamma)M+3)}{R^{\frac{p-1}{p}}((1-2\gamma)M) \left((1-\gamma)^{1-\gamma}\gamma^\gamma\right)^{\frac{1}{(1-\theta)p\gamma}} }\right)^2 +   2 \left(\frac{\sigma^p((3-2\gamma)M+3)^p}{R^{p-1} ((1-2\gamma)M)^p }\right)^{\theta\gamma M}
        \\& \cdot \frac{1}{p\theta\gamma M - 2} \left(\frac{\sigma((3-2\gamma)M+3)}{R^{\frac{p-1}{p}}((1-2\gamma)M) \left((1-\gamma)^{1-\gamma}\gamma^\gamma\right)^{\frac{1}{(1-\theta)p\gamma}} }\right)^{2-p\theta\gamma M}\\
        &= \left(\frac{\sigma((3-2\gamma)M+3)}{R^{\frac{p-1}{p}}((1-2\gamma)M) \left((1-\gamma)^{1-\gamma}\gamma^\gamma\right)^{\frac{1}{(1-\theta)p\gamma}} }\right)^2 \left(1+\frac{2  \left((1-\gamma)^{1-\gamma}\gamma^\gamma\right)^{\frac{\theta M}{(1-\theta)}}  }{p\theta\gamma M -2} \right)\\
        &\leq \left(\frac{\sigma((3-2\gamma)M+3)}{R^{\frac{p-1}{p}}((1-2\gamma)M) \left((1-\gamma)^{1-\gamma}\gamma^\gamma\right)^{\frac{1}{(1-\theta)p\gamma}} }\right)^2 \left(1+\frac{2}{p\theta\gamma M -2} \right),
    \end{align*}
    which completes the proof.
    \end{proof}
    Note that for $M=3$, there exist $\theta\in(0,1)$ and $\gamma\in(0,\frac{1}{2})$ satisfying  $M>\frac{2}{p\theta\gamma}$ if and only if $p>\frac{4}{3}$. For $M=4$, suitable values of $\theta$ and $\gamma$ exist for every $p\in(1,2)$, but the choice must depend on $p$, since $\frac{2}{p\theta\gamma}\to \frac{2}{rq}>4$ as $p\to1$ for any fixed admissible pair. For any $M>4$, a single pair works for all $p\in(1,2)$. Indeed, setting $\theta=\frac{2}{\sqrt{M}}$ and $\gamma=\frac{1}{\sqrt{M}}$ gives $\frac{2}{p\theta\gamma}=\frac{M}{p}<M$. In particular, we may fix $M=5$, $\theta=\frac{9}{10},$ and $\gamma=\frac{4}{9}$ independently of $p$, since $\frac{2}{p\theta\gamma}=\frac{5}{p}<5$ for every $p\in(1,2).$
    Thus, if we know in advance that $p\in(4/3,2)$, we may set $M=3$ and choose $\theta$ and $\gamma$ for the given $p$ to ensure that $\boldsymbol{\nu}_{j^\star}$ has finite second moment. For $p\in(1,4/3]$, the smallest admissible value is $M=4$, with $\theta$ and $\gamma$ again chosen for the given $p$ to have finite second moment of $\boldsymbol{\nu}_{j^\star}$. If $p$ is unknown, we may fix $M=5$, $\theta=\frac{9}{10}$, and $\gamma=\frac{4}{9}$ independently of $p$. This choice ensures that $\boldsymbol{\nu}_{j^\star}$ has finite second moment for every $p\in(1,2).$
    
   Note also that in the expression~\eqref{eqn:truesecondmoment} we see that the bounds cannot be arbitrarily reduced with respect to the number of chunks $M$ because we have
   
    \begin{align*}
        \lim\limits_{M\to\infty} \frac{\sigma((3-2\gamma)M+3)}{R^{\frac{p-1}{p}}((1-2\gamma)M) \left((1-\gamma)^{1-\gamma}\gamma^\gamma\right)^{\frac{1}{(1-\theta)p\gamma}} }=\frac{\sigma(3-2\gamma)}{R^{\frac{p-1}{p}}(1-2\gamma) \left((1-\gamma)^{1-\gamma}\gamma^\gamma\right)^{\frac{1}{(1-\theta)p\gamma}} }
    \end{align*}
    and the second factor goes to $1$ as $M\to\infty.$ However, the significance of $M$ is reflected in the fact that we have lowered the tail probabilities from $O(u^{-p})$ to $O(u^{-p\theta\gamma M}).$ On the other hand, the number of data points per each data chunk $R$ affects the reduction of bound in~\eqref{eqn:truesecondmoment}, namely,
         $\mathbb{E}\left[\|\boldsymbol{\nu}_{j^\star}\|^2 \right]\leq O(R^{-\frac{2(p-1)}{p}}).$ 
         This will be particularly useful when $R$ is allowed to depend on $t$, as this dependence enables us to establish convergence even without the symmetry assumption (see Remark~\ref{rem:increasingsamplesizes} for details).
         
         We close this subsection with the following lemma, which shows that $\boldsymbol{\nu}_{j^\star}$ is a zero mean random variable.

    \begin{lemma}\label{lemma:zeromean}
        Suppose that $\boldsymbol{\nu}_j,$ for $j=1,2,...,M$ are random variables that satisfy Assumption~\ref{as:gradientnoise} and suppose that $\boldsymbol{\nu}_{j^\star}$ is their medoid. Then we have that $\mathbb{E}[\boldsymbol{\nu}_{j^\star}]=0.$
    \end{lemma}
    \begin{proof}
        By Assumption~\ref{as:gradientnoise}, we have $(\boldsymbol{\nu}_1,\boldsymbol{\nu}_2,...,\boldsymbol{\nu}_M)$ and $(-\boldsymbol{\nu}_1,-\boldsymbol{\nu}_2,...,-\boldsymbol{\nu}_M)$ have the same distribution. Therefore,
            $\mathbb{E}[\boldsymbol{\nu}_{j^\star}(\boldsymbol{\nu}_1,\boldsymbol{\nu}_2,...,\boldsymbol{\nu}_M)]=\mathbb{E}[\boldsymbol{\nu}_{j^\star}(-\boldsymbol{\nu}_1,-\boldsymbol{\nu}_2,...,-\boldsymbol{\nu}_M)].$
        Moreover, we have  
        $\boldsymbol{\nu}_{j^\star}(-\boldsymbol{\nu}_1,-\boldsymbol{\nu}_2,...,-\boldsymbol{\nu}_M)=-\boldsymbol{\nu}_{j^\star}(\boldsymbol{\nu}_1,\boldsymbol{\nu}_2,...,\boldsymbol{\nu}_M).$ Hence, $\mathbb{E}[\boldsymbol{\nu}_{j^\star}]=0,$ since
        $\mathbb{E}[\boldsymbol{\nu}_{j^\star}(\boldsymbol{\nu}_1,\boldsymbol{\nu}_2,...,\boldsymbol{\nu}_M)]=-\mathbb{E}[\boldsymbol{\nu}_{j^\star}(\boldsymbol{\nu}_1,\boldsymbol{\nu}_2,...,\boldsymbol{\nu}_M)].$
    \end{proof}

     \section{Main results}\label{sec:main}
     In Subsection~\ref{subsec:expectation} we present the statement and the proof of convergence to first-order stationarity in expectation for R-SGD-Mini algorithm. In Subsection~\ref{subsec:HPB}, we prove that the R-CSGD-Mini algorithm converges with high probability.
    
    \subsection{Convergence to stationarity in expectation}\label{subsec:expectation}
    We now carry out convergence to stationarity in exectation for Algorithm~\ref{alg1}. We have the following theorem.
   \begin{theorem}\label{theorem:convergencetostationarity}
        Let Assumptions~\ref{as:Lipshitzgradient} and~\ref{as:gradientnoise} hold. For $\alpha \leq \frac{1}{2L}$, $\gamma\in(0,\frac{1}{2}),$ $\theta\in(0,1)$ and let $M>\frac{2}{p\theta\gamma}.$, Algorithm~\ref{alg1} generates the sequence of iterates $\{\mathbf{x}^t\}$ such that
        \begin{align*}
        &\frac{1}{T}\sum\limits_{t=0}^{T-1}\mathbb{E}[\|\nabla F(\mathbf{x}^t)\|^2]\leq \frac{2\left(F(\mathbf{x}^0)-F^\star\right) }{\alpha T}\\&+2L\alpha \left(\frac{\sigma((3-2\gamma)M+3)}{R^{\frac{p-1}{p}}((1-2\gamma)M) \left((1-\gamma)^{1-\gamma}\gamma^\gamma\right)^{\frac{1}{(1-\theta)p\gamma}} }\right)^2 \left(1+\frac{2}{p\theta\gamma M -2} \right).
        \end{align*}
    \end{theorem}
    Theorem~\ref{theorem:convergencetostationarity} establishes that for the sequence $\{\mathbf{x}^t\}$ generated by algorithm~\ref{alg1} we have that the expected time-averaged squared gradient norm converges at rate 1/T to a neighborhood of zero. Moreover, Theorem~\ref{theorem:convergencetostationarity} explicitly characterizes that neighborhood through: 1) algorithmic parameters, such as number of data chunks $M$, the number of data points in each data chunk; 2) statistical properties of the gradient noise.
As a direct consequence of Theorem~\ref{theorem:convergencetostationarity}, when the time horizon is known in advance and $\alpha=\mathcal{O}(T^{-\frac{1}{2}}),$ we recover the same rate as SGD in the light-tailed setting, i.e. $ \frac{1}{T}\sum\limits_{t=0}^{T-1}\mathbb{E}[\|\nabla F(\mathbf{x}^t)\|^2]\leq \mathcal{O}(T^{-\frac{1}{2}}).$ It can be seen that this holds even when gradient noise has heavy tailed distribution. In contrast, it can be shown that 
that standard SGD diverges when gradient noise has infinite variance (see, for example, Appendix B in~\cite{dusan}). Quantity $ \frac{1}{T}\sum\limits_{t=0}^{T-1}\mathbb{E}[\|\nabla F(\mathbf{x}^t)\|^2]$
is a standard metric for non-convex stochastic optimization, see e.g.,~\cite{ghadimi}. It equals the expected squared gradient norm at the past iterate $\mathbf{x}^s$ selected uniformly at random among the iterates seen so far $\mathbf{x}^0,....,\mathbf{x}^{T-1}.$ 
We now prove Theorem~\ref{theorem:convergencetostationarity}.
\begin{proof}
        We have that
        \begin{align*}
            F(\mathbf{x}^t-\alpha\nabla &f_j(\mathbf{x}^t))\leq F(\mathbf{x}^t) -\alpha \nabla F(\mathbf{x}^t)^\top \nabla f_j(\mathbf{x}^t)+\alpha^2 \frac{L}{2}\|\nabla f_j(\mathbf{x}^t)\|^2\\
            &\leq F(\mathbf{x}^t) - \alpha \nabla F(\mathbf{x}^t)^\top (\nabla F(\mathbf{x}^t)+\boldsymbol{\nu}_j^t(\mathbf{x}^t))\\&+L\alpha^2\left(\|\nabla F(\mathbf{x}^t)||^2+\|\boldsymbol{\nu}_j^t(\mathbf{x}^t)\|^2\right)\\
            &= F(\mathbf{x}^t)- \frac{\alpha}{2}\|\nabla F(\mathbf{x}^t)\|^2 -\alpha \nabla F(\mathbf{x}^t)^\top \boldsymbol{\nu}_j^t (\mathbf{x}^t)+L\alpha^2 \|\boldsymbol{\nu}_j^t(\mathbf{x}^t)\|^2,
        \end{align*}
        where we have used equations~\eqref{eqn:gradnoise}, $(a+b)^2\leq 2a^2+2b^2$ and $\alpha\leq\frac{1}{2L}.$ 
        Next, by algorithm construction we have
        \begin{align}\label{eqn:modification}
            F(\mathbf{x}^t-\alpha\nabla f_{j^\star}^t(\mathbf{x}^t))&= F(\mathbf{x}^t)- \frac{\alpha}{2}\|\nabla F(\mathbf{x}^t)\|^2 -\alpha \nabla F(\mathbf{x}^t)^\top \boldsymbol{\nu}_{j^\star}^t (\mathbf{x}^t)+L\alpha^2 \|\boldsymbol{\nu}_{j^\star}^t(\mathbf{x}^t)\|^2.
        \end{align}
        Setting that $\mathbb{E}_t[\cdot]=\mathbb{E}[\cdot|\mathcal{F}_t]$  and utilizing Theorem~\ref{theo:variance} and Lemma~\ref{lemma:zeromean} we have that
         \begin{align}
            \mathbb{E}_t\left[  F(\mathbf{x}^t-\alpha\nabla f_{j^\star}^t(\mathbf{x}^t))\right]& \leq F(\mathbf{x}^t) - \frac{\alpha}{2}\|\nabla F(\mathbf{x}^t)\|^2 -\alpha \nabla F(\mathbf{x}^t)^\top \mathbb{E}_t\left[\boldsymbol{\nu}_{j^\star}^t\right]\nonumber \\&+ L\alpha^2\mathbb{E}_t\left[ \|\boldsymbol{\nu}_{j^\star}^t\|^2\right]\nonumber\\
            &\leq F(\mathbf{x}^t) - \frac{\alpha}{2}\|\nabla F(\mathbf{x}^t)\|^2+ L\alpha^2 B,\label{eqn:prefinal}
        \end{align}
        where $B=\left(\frac{\sigma((3-2\gamma)M+3)}{R^{\frac{p-1}{p}}((1-2\gamma)M) \left((1-\gamma)^{1-\gamma}\gamma^\gamma\right)^{\frac{1}{(1-\theta)p\gamma}} }\right)^2 \left(1+\frac{2}{p\theta\gamma M -2} \right)$ is given by~\eqref{eqn:truesecondmoment}.
        Finally, by taking the expectation in~\eqref{eqn:prefinal}, we have the following
        \begin{align*}
            \mathbb{E}[F(\mathbf{x}^{t+1})]&\leq \mathbb{E}[F(\mathbf{x}^t)]-\frac{\alpha}{2} \mathbb{E}[\|\nabla F(\mathbf{x}^t)\|^2] +  L\alpha^2 B,
        \end{align*}
        hence, $\frac{\alpha}{2}\sum\limits_{t=0}^{T-1}\mathbb{E}[\|\nabla F(\mathbf{x}^t)\|^2]\leq F(\mathbf{x}^0)- F^\star + TL\alpha^2 B$
        which completes the proof.
    \end{proof}

\begin{remark}\label{rem:assymetriextenison}
    Although symmetric gradient noise has been widely adopted in previous works~\cite{chen2020understanding,bernstein2018signsgd,barsbey2021heavy,battash2024revisiting,armacki2025high,bernstein2018signsgd2}, this assumption can be relaxed. Specifically, we use an alternative set of assumptions that enables the symmetrization procedure proposed in~\cite{armacki}. This set of assumptions requires, among other things, that at the start of training we have access to a noiseless gradient of $F$ at any reference point, and that the noises are i.i.d. for two different samples evaluated at two different points. See Appendix~\ref{app:assymetriextenison} for details.
\end{remark}
\begin{remark}\label{rem:increasingsamplesizes}
    If $p$ is known in advance and sufficient computational resources are available, we can obtain a tighter bound than the one in Theorem~\ref{theorem:convergencetostationarity}, even when the gradient noise is asymmetric. Specifically, if we consider a variant of Algorithm~\ref{alg1} in which the sample size increases with the iteration index $t$ (recall that we refer to this variant as R-iSGD-Mini) i.e., if we use $K_t = M R_t$ samples at iteration $t$, where $R_t = (t+1)^n$, we get that 
    $\frac{1}{T}\sum\limits_{t=0}^{T-1}\mathbb{E}[\|\nabla F(\mathbf{x}^t)\|^2]=\mathcal O\!\left(T^{-\frac{2n(p-1)}{p}}\right)$ if  $n<\dfrac{p}{2(p-1)}$, $\frac{1}{T}\sum\limits_{t=0}^{T-1}\mathbb{E}[\|\nabla F(\mathbf{x}^t)\|^2]=\mathcal O\!\left(\dfrac{\log T}{T}\right)$ if $n=\dfrac{p}{2(p-1)}$ and $\frac{1}{T}\sum\limits_{t=0}^{T-1}\mathbb{E}[\|\nabla F(\mathbf{x}^t)\|^2]= \mathcal O\!\left(\dfrac{1}{T}\right),$ if $n>\dfrac{p}{2(p-1)}.$ Thus, by increasing the sample size with the iteration, we can achieve a convergence rate faster than existing lower bounds, under both symmetric and asymmetric gradient noise.
\end{remark}
\begin{remark}\label{rem:diminishingextension}
Algorithm~\ref{alg1} can also be used with diminishing step sizes, i.e., $\alpha_t = \frac{a}{(t+1)^{\frac{1}{2}}}$ for $a>0$. In this case, we get
$\min\limits_{0\leq k\leq t-1} \mathbb{E}[\|\nabla F(\mathbf{x}^k)\|^2]=\mathcal{O}(\frac{\log t}{\sqrt{t}})$. See Appendix~\ref{app:diminishingextension} for details.
\end{remark}

\subsection{Convergence with High-Probability Bounds}\label{subsec:HPB}
In this subsection, we show that the medoid gradient estimator can be readily incorporated into existing frameworks, such as those of~\cite{nguyen2023improved} and~\cite{armacki}, thereby allowing us to inherit their convergence guarantees.
\subsubsection{High-Probability Bounds with Clipping}
We consider the setting in which clipping is applied to the proposed gradient estimator, i.e., Step 5 in Algorithm~\ref{alg1} is modified such that $\mathbf{g}^t=\min\{1,\frac{\lambda_t}{\|\nabla f_{j^\star}^t(\mathbf{x}^t)\|}\}\nabla f_{j^\star}^t(\mathbf{x}^t).$ Recall that we refer to the resulting algorithm as R-CSGD-Mini algorithm. Under the assumption that the noise is symmetric, the medoid noise has zero mean. Furthermore, Theorem~\ref{theo:variance} guarantees that this noise has a finite second moment, whenever $M>\frac{2}{p\theta\gamma}$ for fixed $\gamma$ and $\theta$. Therefore, all the assumptions of Theorem 5.1 from~\cite{nguyen2023improved} (a restatement of Theorem is provided in Appendix~\ref{app:theoHPB}, Theorem~\ref{theo:nguyen}) are satisfied with $p=2$, implying that R-CSGD-Mini achieves a convergence rate of $\mathcal{O}(T^{-\frac{1}{2}})$ with high probability. In Example~\ref{ex:vsnguyen} (see Section~\ref{sec:examples}), we compare the bounds obtained from Theorem~\ref{theo:nguyen} when using the average gradient computed from $K$ samples and when partitioning the samples into chunks and computing their medoid. The results show that, for sufficiently heavy-tailed noise, the medoid-based approach yields a tighter bound even for short time horizons (as expected, the bound is tighter for sufficiently large horizons because it decays faster). 

\subsubsection{High-Probability Bounds with a General Nonlinearity}
We consider the setting in which a general nonlinearity $\boldsymbol{\Psi}: \mathbb{R}^d \mapsto \mathbb{R}^d$ from~\cite{armacki} is applied to the proposed gradient estimator $\nabla f_{j^\star}^t(\mathbf{x}^t)$, i.e., Step 5 in Algorithm~\ref{alg1} is modified such that $\mathbf{g}^t=\boldsymbol{\Psi}(\nabla f_{j^\star}^t(\mathbf{x}^t)).$ 
Recall that the resulting algorithm is denoted by R-NSGD-Mini. Here, nonlinearity $\boldsymbol{\Psi}$ satisfies Assumption 6 from~\cite{armacki}, which we restate as Assumption~\ref{asmpt:nonlin-state} in Appendix~\ref{app:armacki}. 
To obtain a high-probability guarantee of the same form as in~\cite{armacki}, we additionally assume that the gradient noise density is bounded away from zero in a neighborhood of the origin. In Appendix~\ref{app:armacki}, we show that under this additional assumption, the noise associated with the medoid gradient estimator also has a density that is bounded away from zero in a neighborhood of the origin. 
Thus, the conditions of Theorem 1 from~\cite{armacki} (see also Theorem~\ref{thm:n-sgd} in Appendix~\ref{app:armacki}) are satisfied.
In the preceding subsection, applying the result of~\cite{nguyen2023improved} yields a convergence rate of $\mathcal{O}(T^{-\frac{1}{2}})$ for the medoid gradient estimator, whereas the rate for the standard averaged gradient estimator is slower. In the present setting, Theorem 1 from~\cite{armacki} already guarantees the optimal rate with the standard averaged gradient estimator. However, as Example~\ref{ex:vsarm} illustrates, applying the nonlinearity to the medoid gradient estimator yield a tighter upper bound.

\section{Analytical and Numerical examples}	\label{sec:examples}
\begin{example}\label{ex:vsnguyen}
We compare the bound from Theorem~\ref{theo:nguyen} (see Appendix~\ref{app:theoHPB}, see also~\cite{nguyen2023improved}) for the two methods: when clipping is applied to the average gradient computed from $K$ data samples, and when it is applied to the medoid  gradient of $M$ data chunks. We set the Lipschitz constant to $L=1,$ use the same initial optimality gap $\Delta=1$ and failure probability $\delta=0.1$, and set $\sigma=1$ in the $p$-th moment bound on the gradient noise associated with a single data sample.
The resulting value of $\sigma$ for $K$ samples scales accordingly, whereas $\sigma_{j^\star}$ for the medoid chunk gradient noise it is given in Theorem~\ref{theo:variance}, for $M=5, \theta=\frac{9}{10}$ and $\gamma=\frac{4}{9},$ and thus, the effective value of $p$ for medoid chunk gradient noise is 2.
Figure~\ref{fig:nguyen} shows the bound from Theorem~\ref{theo:nguyen} as a function of $T$ when the noise has a bounded $p$-th moment, for $p=1.05$ and $p=1.1$. For $p=1.05$, the medoid performs better from the outset and maintains its advantage. In contrast, for $p=1.1$, the average gradient initially yields a tighter bound; nevertheless, the medoid eventually outperforms it due to its more favorable convergence rate. Thus, for small values of $p$, corresponding to very heavy-tailed noise, R-CSGD-Mini yields a tighter bound from the beginning. When the noise is moderately heavy-tailed, its faster convergence rate yields a tighter bound for sufficiently large time horizons.
\begin{figure}
		\centering \includegraphics[height=60mm]{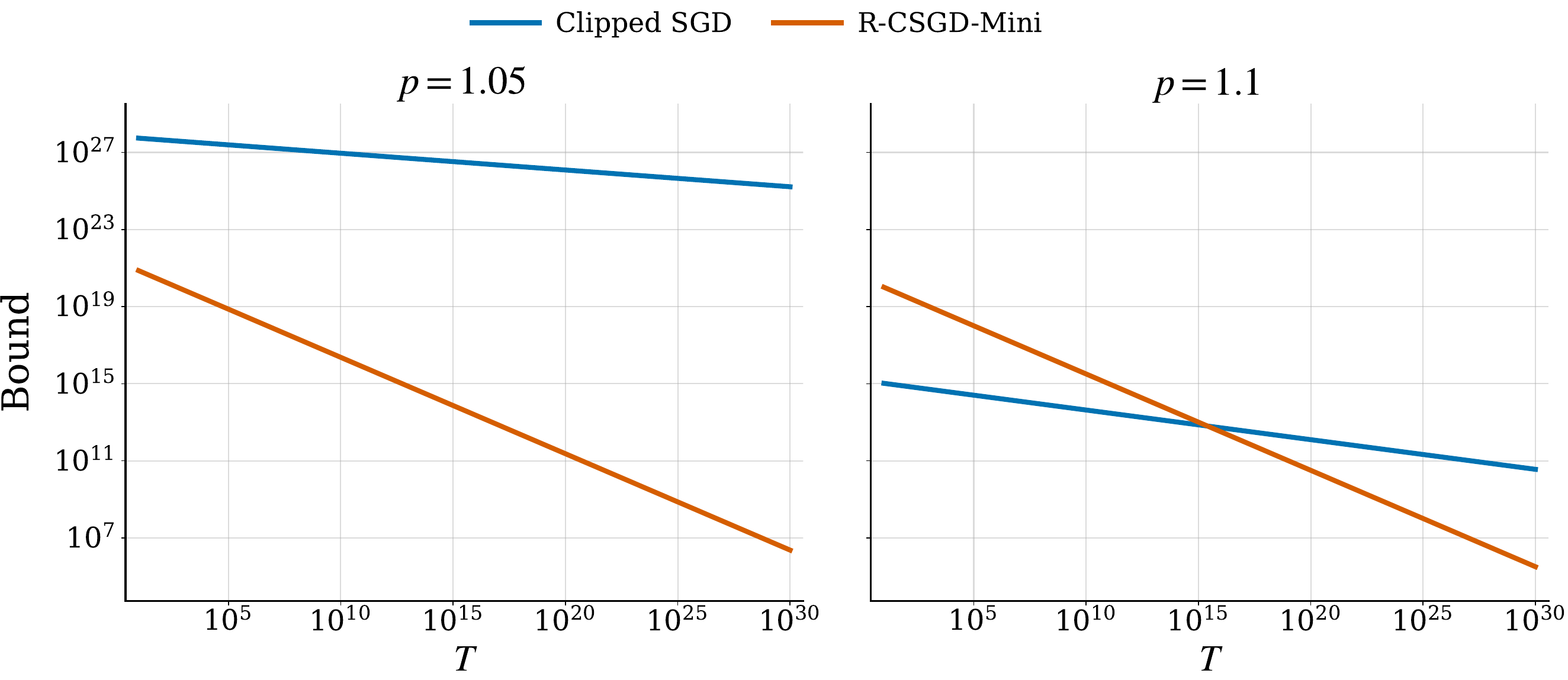}
		\caption{ Bounds from Theorem~\ref{theo:nguyen} for $p = 1.05$ and $p = 1.1$ when clipping is applied to average gradient and medoid of chunk gradients}\label{fig:nguyen}
	\end{figure}
\end{example}

\begin{example}\label{ex:vsarm}
In this example, we compare the bounds in Theorem~\ref{thm:n-sgd} for fixed step size, i.e.,
\begin{align}\label{eqn:armbound}
	\frac{4(log(\frac{1}{\delta})+\Delta)}{a\beta}+\frac{2aLC^2}{\beta},
\end{align}
obtained by applying the nonlinearity to the average gradient and to the medoid of the gradient estimates computed from $M=5$ chunks. We use the componentwise nonlinearity $\psi(x)=\tanh x,$ for which $C=\sqrt{d}$ and $K_2=\frac{4\sqrt{3}}{9}.$ As in the previous example, we consider a problem with $\Delta=1$, $L=1$, and $\tilde{\nabla}=1.0$, and set $\delta=0.1.$ we model the noise using a radial probability density function $p(\boldsymbol{\nu})=\frac{\Gamma(d+\kappa-1)}
{S_{d-1}\Gamma(d)\Gamma(\kappa-1)}
(1+||\boldsymbol{\nu}||)^{-(\kappa+d-1)},$ where $\Gamma(\cdot)$ denotes the Gamma function, and
$S_{d-1}$
denotes the surface area of the unit sphere in $\mathbb{R}^d$.
Note that, unlike the bound in Theorem~\ref{theo:nguyen} considered in the previous example, bound~\eqref{eqn:armbound} from Theorem~\ref{thm:n-sgd} does not depend explicitly on an upper bound for the $p$th moment of the noise. It does, however, involve a parameter $\beta$ that depends on the noise distribution. More precisely, we have $\beta_1 = \frac{(\phi_2^\prime(0))^2}{4K_2\sqrt{d}},$ $\beta_2 = \frac{\phi^\prime(0)}{2d},$ and $\beta = \min\{\beta_1,\beta_2\}$ where 
\begin{align}\label{eqn:phiprime}
    \phi^\prime(0) = \int_{-\infty}^{\infty}\psi_2^\prime(z)\rho(z)dz = 1 - \int_{-\infty}^{\infty}\tanh^2(z)\rho(z)dz
\end{align}
and $\rho$ is underlying probability density function (see also Example 2 in~\cite{armacki}). 
To estimate the quantities in~\eqref{eqn:phiprime}, we used Monte Carlo simulation. In each replication, we drew $K$ independent noise samples from the distribution with probability density function $p$. Using the same samples, we computed the sample mean and the medoid of the $M$ chunk means. For each estimator, we evaluated $\tanh^2(\nu)$ at its output $\nu$ and averaged the resulting values across replications. In the first experiment, we investigated its dependence on the dimension $d$ by evaluating it at 100 approximately equally spaced integer values between $100$ and $1000,$ while keeping all other quantities fixed and for $\kappa=2.00001.$
In the second experiment, we investigated the dependence of the bound on the tail index $\kappa$, where $\kappa$ was varied uniformly over the interval $[2.001,3]$ using $100$ equally spaced values. The dimension was fixed at $d=1000.$

Figure~\ref{fig:arm} shows that as $\kappa$ approaches $2$, the bound of N-SGD increases dramatically, reaching values on the order of $10^{18},$ whereas the bound of R-NSGD-Mini remains an order of magnitude smaller. On the other hand, when the noise is only mildly heavy-tailed, the bound for R-NSGD-Mini exceeds that for N-SGD once $\kappa$ crosses a certain threshold.
Figure~\ref{fig:arm} also shows that the bound for N-SGD grows rapidly with dimension, reaching values on the order of $10^{18},$ while the bound for R-NSGD-Mini remains an order of magnitude smaller.
Thus, as demonstrated, R-NSGD-Mini can drastically reduce the convergence constant and its dependence on the dimension and tail index.

\begin{figure}
		\centering \includegraphics[height=60mm]{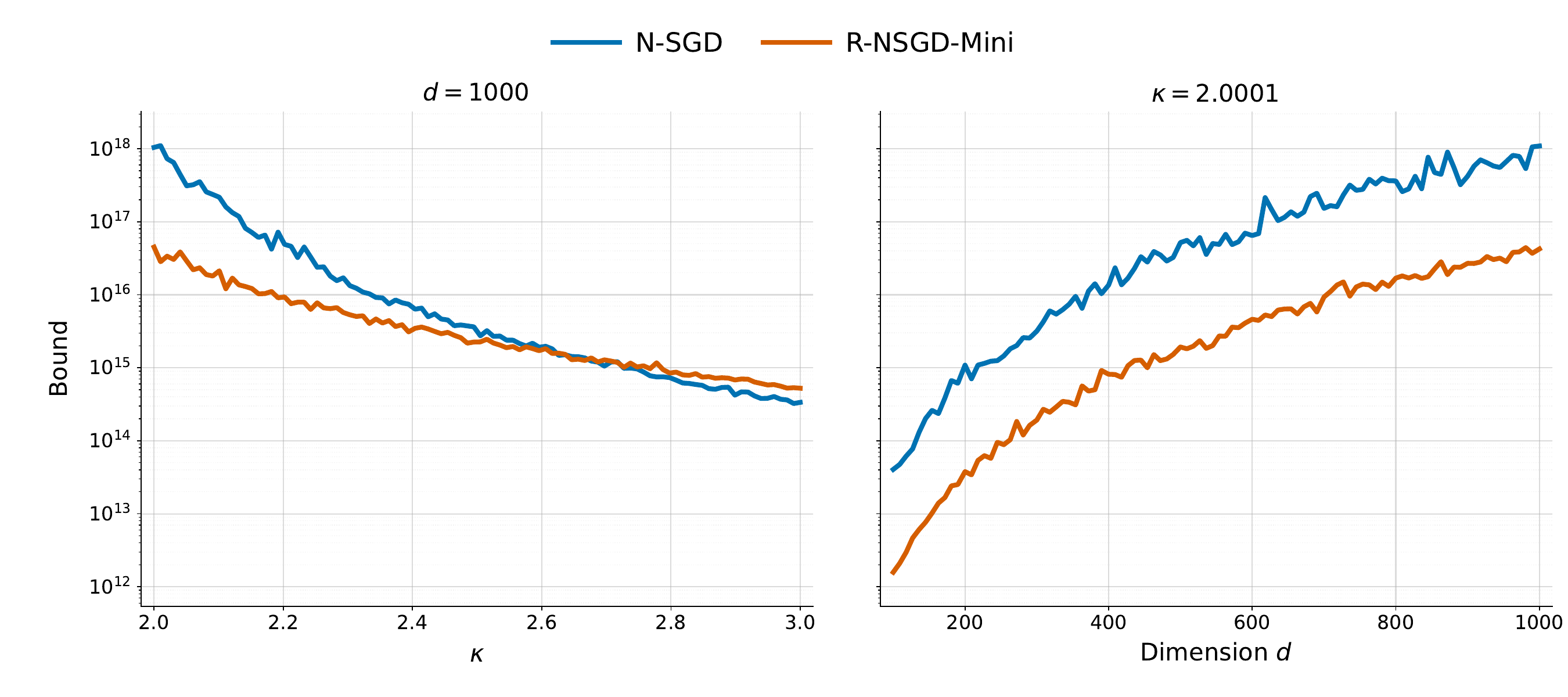}
		\caption{ Bound~\eqref{eqn:armbound} from Theorem~\ref{thm:n-sgd} for N-SGD and R-NSGD-Mini with respect to dimension $d$ and and tail index $\kappa$ }\label{fig:arm}
	\end{figure}

\end{example}

\begin{example}\label{ex:mo}
We compared SGD method when gradient estimator is coordinatewise median-of-means (MoM), geometric median-of-means (GMoM), and medoid-of-means (R-SGD-Mini) on the synthetic objective
$F(\mathbf{x})=\tanh^2(\mathbf{x}_1)+\sum_{i=2}^{d}\mathbf{x}_i^2,$ where
$\mathbf{x}\in\mathbb R^d,$
with $d=400$. The objective has a globally $2$-Lipschitz continuous gradient and a unique global minimizer at $\mathbf{x}^\star=0.$ All methods were initialized at $\mathbf{x}_1=0.5$ and $\mathbf{x}_i=10/\sqrt{d-1}$ for $i=2,\ldots,d$. We used a constant step size $\alpha=10^{-3}$ and ran each method for $T=5{,}000$ updates in each of $10$ independent Monte Carlo simulations.

At iteration $t$, we generated $K=100$ stochastic gradient samples of the form $ 
\nabla F(\mathbf{x}^t)+\boldsymbol{\nu}^t,$
where the noise vectors $\boldsymbol{\nu}^t$ were drawn independently from the rotationally symmetric density
$p(\boldsymbol{\nu})=\frac{\Gamma(d+\kappa-1)}
{S_{d-1}\Gamma(d)\Gamma(\kappa-1)}
(1+||\boldsymbol{\nu}||)^{-(\kappa+d-1)},$ for $ \kappa=2.01,$
and $S_{d-1}$ is the surface area of the unit sphere in $\mathbb R^d$. This choice gives mean-zero noise with an infinite second moment. We divided the $K$ samples into $M=5$ data chunks of size $R=20.$

The coordinatewise MoM estimator takes the median of the chunks separately in each coordinate. The GMoM estimator is given in~\ref{rem:choicePsi} and we approximated it by using the modified Weiszfeld algorithm~\cite{weiszfeld}, initialized at the coordinatewise MoM estimate, with a stopping tolerance of \(10^{-10}\). We tested two iteration budgets, allowing at most 2 and 5 Weiszfeld iterations per outer update. Recall that R-SGD-Mini selects the gradient of the chunk with the smallest sum of Euclidean distances to the other chunks gradients.

Figure~\ref{fig:om} shows the Monte Carlo average of the running minimum of the squared norm of the exact gradient, i.e.,
$\frac{1}{10}\sum_{r=1}^{10}
\min_{0\leq t\leq T}\|\nabla F(\mathbf{x}_r^{t})\|_2^2,$
on a logarithmic vertical scale. The stochastic gradient estimates determine the iterates $\mathbf{x}_r^{t}$. We evaluated the plotted metric directly from the analytic gradient at those iterates.

The average runtimes were 1.4306s for coordinatewise MoM, 1.6363s for GMoM with at most 2 Weiszfeld iterations,  1.9573s for GMoM with at most 5 iterations, and 1.3322s for R-SGD-Mini.

In this experiment, increasing the Weiszfeld budget from 2 to 5 iterations yielded a lower running minimum of the squared gradient norm at the cost of a longer runtime. The smaller budget can leave the estimate farther from the exact geometric median because the inner problem is solved less accurately. The R-SGD-Mini was faster than both GMoM variants and coordinatewise MoM. Moreover it uses the joint geometry of the chunks through Euclidean distances, whereas coordinatewise MoM aggregated each coordinate separately. Therefore, it can be seen that R-SGD-Mini therefore offered a favorable trade-off between the achieved gradient norm and runtime.

    \begin{figure}
		\centering \includegraphics[height=60mm]{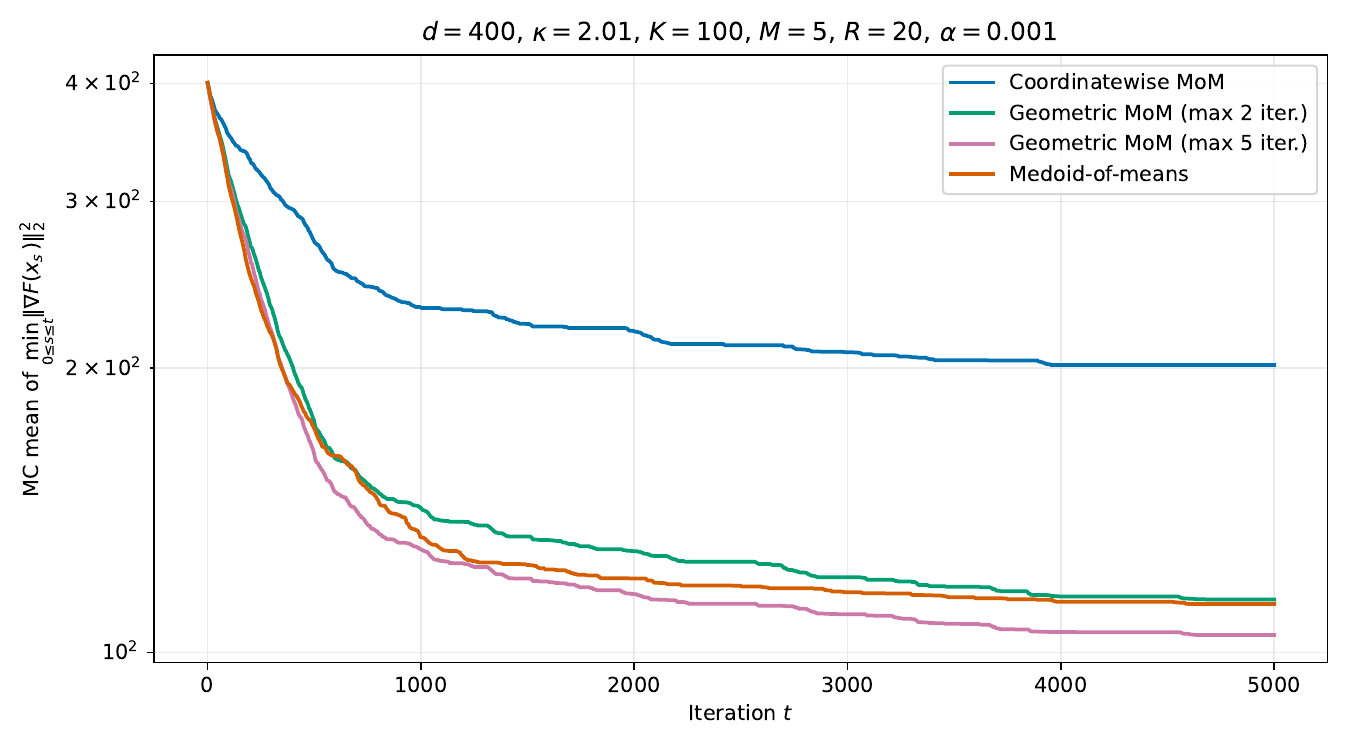}
		\caption{ Monte Carlo average of
$
\min_{0\leq t\leq T}\|\nabla F(\mathbf{x}^t)\|_2^2,$ for different choice of gradient estimator}\label{fig:om}
	\end{figure}
\end{example}

\begin{example}\label{ex:increasingsamplesizes}
    We compared deterministic gradient descent with two constant-stepsize stochastic methods on $F(x)=\tanh^2(x_1)+\sum_{i=2}^{d}x_i^2$ for $d=500$: SGD with the arithmetic mean and R-iSGD-Mini. At update $t$, both stochastic methods used the same $K_t=MR_t$ noisy gradient samples (as in Remark~\ref{rem:increasingsamplesizes}, with noise drawn from the specified radial distribution with $\kappa=2.4.$ We used $M=5$ chunks of size $R_t=(t+1)^n$ for $n=1,2$. By Remark~\ref{rem:increasingsamplesizes}, the bound is of order \(\mathcal{O}(T^{-1})\) whenever $n>\dfrac{p}{2(p-1)}.$ For the noise distribution considered here, the $p$-th moment is finite for $p<\kappa-1=1.4$. Thus, choosing $p=1.39,$ we see that $n=2$ satisfies the condition for the $\mathcal{O}(T^{-1})$ bound, while $n=1$ yields the slower rate given in Remark~\ref{rem:increasingsamplesizes}. For each stochastic method, we averaged the squared norm of the exact gradient across Monte Carlo runs at each iterate and plotted the running minimum of these averages. Figure~\ref{fig:inc} shows that, for $n=2$, the R-iSGD-Mini curve has a slope similar to that of deterministic gradient descent over the plotted range, whereas SGD with the arithmetic mean and R-iSGD-Mini for $n=1$ shows less improvement in this experiment. The running minimum makes the curves nonincreasing by hiding temporary increases in the squared gradient norm. For a horizon of $T=200$, the total number of samples per run is $99,500$ and $13,233,500$ for $n=1$ and $n=2,$ respectively. Thus, obtaining a convergence bound comparable to that of deterministic gradient descent and observing similar behavior in practice requires a large sample budget, which substantially increases the computational cost and runtime.
    \begin{figure}
		\centering \includegraphics[height=60mm]{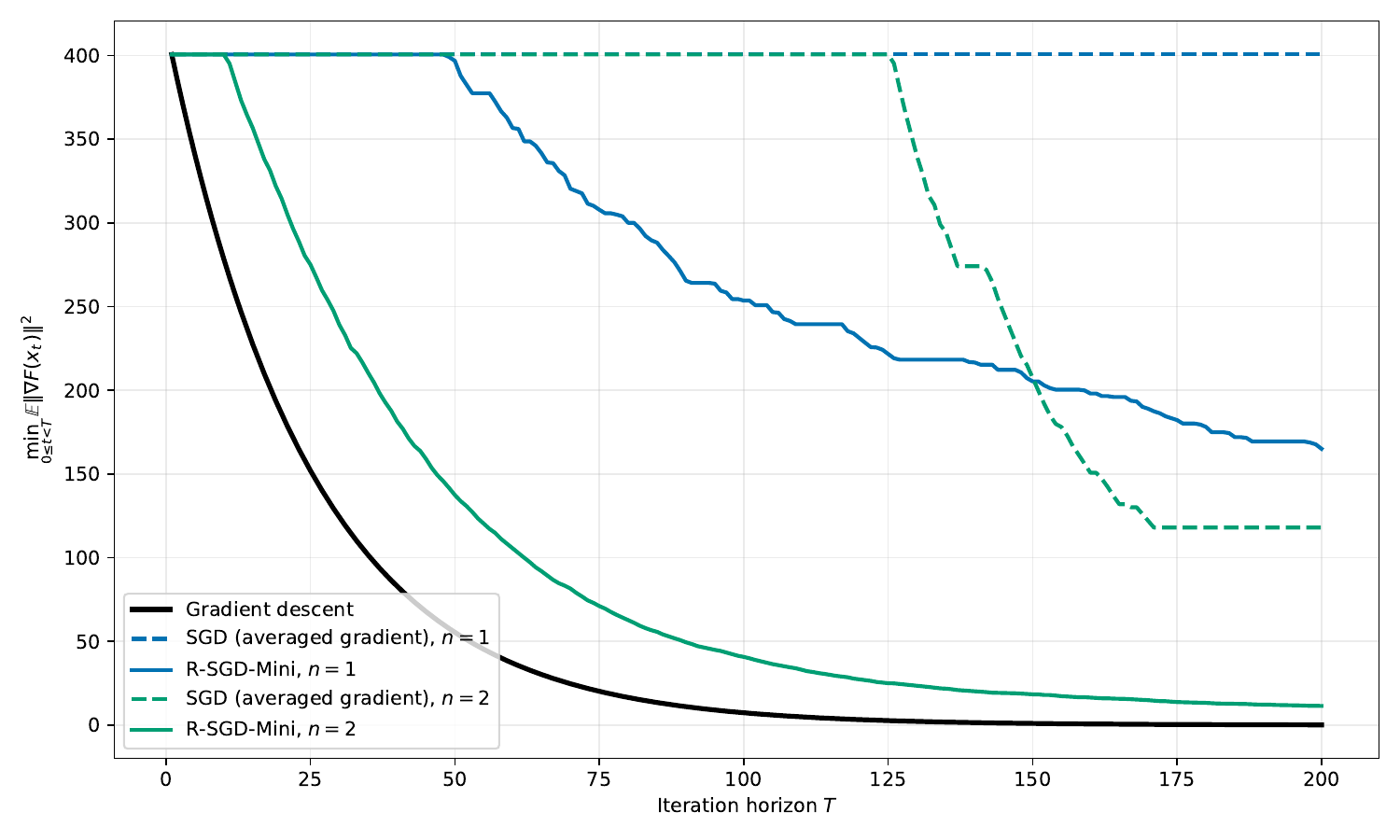}
		\caption{ Monte Carlo average of
$
\min_{0\leq t\leq T}\|\nabla F(\mathbf{x}^t)\|_2^2,$ for GD, SGD and R-iSGD-Mini when $K_t=MR_t$ noisy gradient samples are used}\label{fig:inc}
	\end{figure}
\end{example}

 \paragraph{Real-Data Experiment: CIFAR-10 (ResNet-18).} We evaluate the behavior of several optimization methods on the CIFAR-10~\cite{cifar10} benchmark using a modified ResNet-18~\cite{resnet} with a 3×3 (stride 1) initial convolution and no max-pooling layer. We use the standard train/test split and apply standard normalization with mean $(0.4914, 0.4822, 0.4465)$ and standard deviation $(0.2023, 0.1994, 0.2010)$ to all inputs. Training is performed in an iteration-based regime, where a fixed sequence of batches is pre-generated and reused across all optimizers to ensure a controlled and fair comparison. We compared R-SGD-Mini and R-CSGD-Mini with standard stochastic gradient descent (SGD), its clipped variant~\cite{HTASGD,CSGDHT}, as well as robust alternative based on Median-of-Means (MoM) and their clipped variant~\cite{puchkinMedian}.  
 The experiments are implemented in Python $3.10.20$ using PyTorch $2.7.1$ and executed on NVIDIA Tesla T4 GPU and conducted over $5$ Monte Carlo runs, each consisting of $20,000$ iterations, using a learning rate of $0.1$ and a batch size of $256$. Final-iteration results are reported in Table~\ref{table:summary} and accuracy over iterations is shown in Figure~\ref{fig:accuracy}. Here, $\lambda$ denotes the clipping parameter, and $M$ is the number of data chunks.



\begin{table}[h]
\centering
\newcommand{\meanstd}[2]{$#1 \pm #2$}

\resizebox{\linewidth}{!}{%
\begin{tabular}{@{}lcccccc@{}}
\toprule
Optimizer & $\lambda$ & $M$ & Accuracy & Train Loss & Test Loss & Time (s) \\
\midrule
R-SGD-Mini & -- & 3
    & \meanstd{0.8141}{0.0037}
    & \meanstd{1.31 \times 10^{-2}}{1.07 \times 10^{-3}}
    & \meanstd{1.02}{9.36 \times 10^{-3}}
    & \meanstd{5780.28}{5.27} \\
R-SGD-Mini & -- & 4
    & \meanstd{0.8128}{0.0032}
    & \meanstd{4.74 \times 10^{-2}}{6.26 \times 10^{-3}}
    & \meanstd{1.01}{3.27 \times 10^{-2}}
    & \meanstd{7162.86}{8.05} \\
R-CSGD-Mini & 5 & 3
    & \meanstd{0.8126}{0.0047}
    & \meanstd{1.32 \times 10^{-2}}{1.65 \times 10^{-3}}
    & \meanstd{1.00}{1.76 \times 10^{-2}}
    & \meanstd{5785.12}{17.17} \\
R-CSGD-Mini & 5 & 4
    & \meanstd{0.8101}{0.0027}
    & \meanstd{4.93 \times 10^{-2}}{6.86 \times 10^{-3}}
    & \meanstd{1.03}{2.64 \times 10^{-2}}
    & \meanstd{7168.36}{22.71} \\
Clipped SGD & 5 & --
    & \meanstd{0.7717}{0.0039}
    & \meanstd{5.45 \times 10^{-5}}{1.96 \times 10^{-6}}
    & \meanstd{1.15}{4.73 \times 10^{-2}}
    & \meanstd{2691.38}{5.71} \\
Clipped MoM & 5 & 3
    & \meanstd{0.7716}{0.0035}
    & \meanstd{6.53 \times 10^{-5}}{1.47 \times 10^{-6}}
    & \meanstd{1.16}{2.88 \times 10^{-2}}
    & \meanstd{6577.20}{12.79} \\
SGD & -- & --
    & \meanstd{0.7706}{0.0045}
    & \meanstd{5.45 \times 10^{-5}}{3.35 \times 10^{-6}}
    & \meanstd{1.15}{2.49 \times 10^{-2}}
    & \meanstd{2686.67}{4.68} \\
MoM & -- & 4
    & \meanstd{0.7689}{0.0056}
    & \meanstd{8.69 \times 10^{-5}}{2.31 \times 10^{-5}}
    & \meanstd{1.17}{3.20 \times 10^{-2}}
    & \meanstd{7527.57}{27.37} \\
Clipped MoM & 5 & 4
    & \meanstd{0.7678}{0.0023}
    & \meanstd{7.95 \times 10^{-5}}{2.75 \times 10^{-5}}
    & \meanstd{1.17}{2.86 \times 10^{-2}}
    & \meanstd{7515.14}{12.68} \\
MoM & -- & 3
    & \meanstd{0.7668}{0.0042}
    & \meanstd{6.69 \times 10^{-5}}{3.92 \times 10^{-6}}
    & \meanstd{1.19}{3.26 \times 10^{-2}}
    & \meanstd{6585.18}{18.00} \\
Clipped SGD & 2 & --
    & \meanstd{0.7278}{0.0091}
    & \meanstd{1.97 \times 10^{-3}}{3.89 \times 10^{-3}}
    & \meanstd{1.65}{6.95 \times 10^{-2}}
    & \meanstd{2694.01}{4.50} \\
Clipped MoM & 2 & 3
    & \meanstd{0.7235}{0.0051}
    & \meanstd{2.85 \times 10^{-5}}{1.42 \times 10^{-6}}
    & \meanstd{1.59}{3.68 \times 10^{-2}}
    & \meanstd{6584.81}{12.28} \\
Clipped MoM & 2 & 4
    & \meanstd{0.7210}{0.0072}
    & \meanstd{3.34 \times 10^{-5}}{1.19 \times 10^{-6}}
    & \meanstd{1.56}{5.02 \times 10^{-2}}
    & \meanstd{7512.41}{7.54} \\
R-CSGD-Mini & 2 & 3
    & \meanstd{0.7187}{0.0103}
    & \meanstd{1.67 \times 10^{-1}}{3.57 \times 10^{-2}}
    & \meanstd{1.70}{3.45 \times 10^{-2}}
    & \meanstd{5771.56}{5.27} \\
R-CSGD-Mini & 2 & 4
    & \meanstd{0.7111}{0.0247}
    & \meanstd{2.45 \times 10^{-1}}{8.99 \times 10^{-2}}
    & \meanstd{1.63}{2.14 \times 10^{-1}}
    & \meanstd{7167.87}{8.81} \\
\bottomrule
\end{tabular}%
}

\caption{Performance comparison across optimizers. Results are reported
as mean $\pm$ standard deviation and sorted by decreasing mean accuracy.}
\label{table:summary}
\end{table}

\begin{figure}[tbhp]
		\centering \includegraphics[height=60mm]{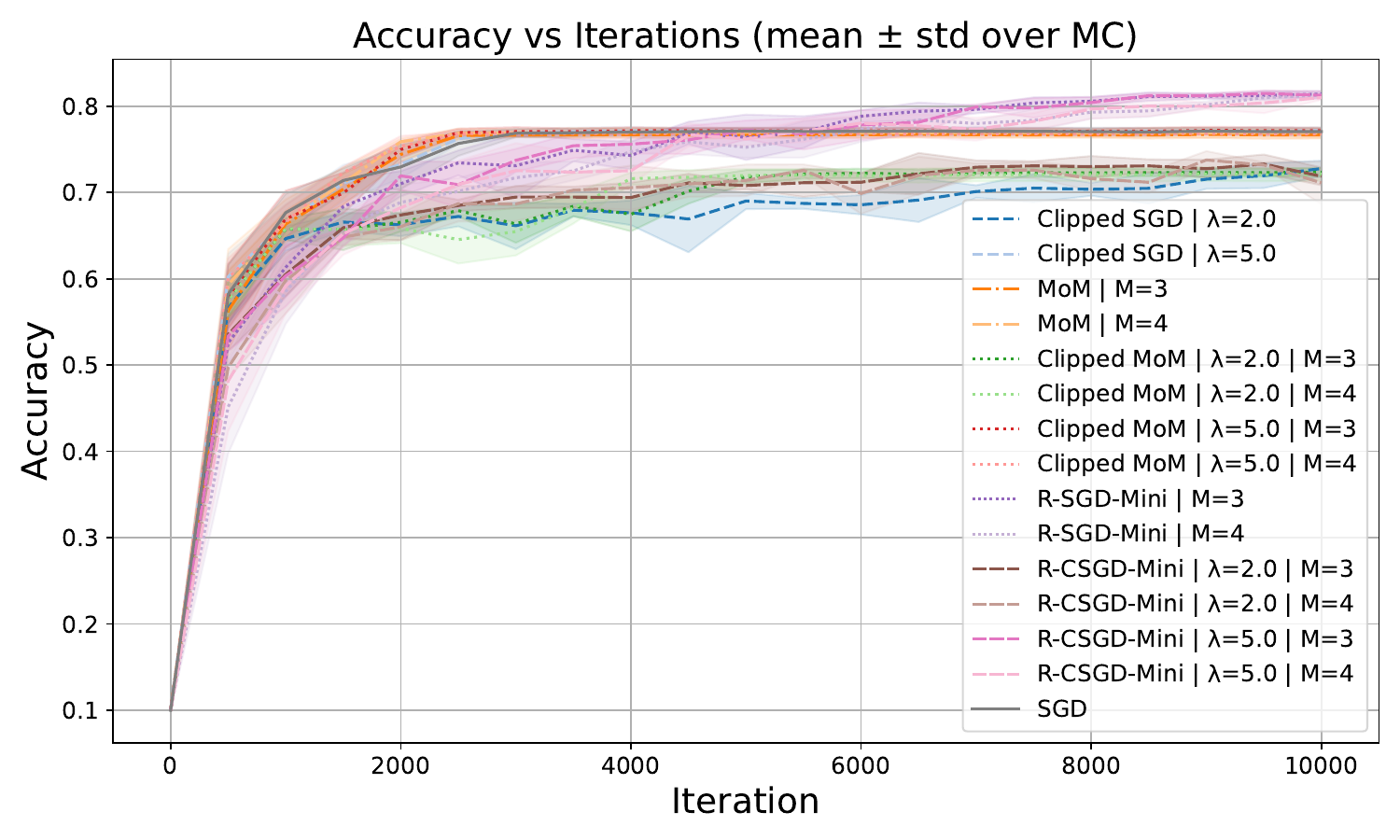}
		\caption{ Test Accuracy over Iterations on CIFAR-10 (ResNet-18)}\label{fig:accuracy}
	\end{figure}

From Table~\ref{table:summary}, it can be seen that R-SGD-Mini and R-CSGD-Mini achieve the highest accuracy, averaged over Monte Carlo runs. The best among them reaches 0.8141, while the best among the remaining methods achieves 0.7717.
It is also worth noting that, although the accuracy is the highest, the training loss is not. This is because the algorithm is designed to ignore the most extreme gradients, which are associated with the hardest data points. In contrast, the other methods take these points into account, which degrades accuracy, particularly for samples near the decision boundary.
Clipping at $\lambda=2$ worsens performance across all algorithms, including R-CSGD-Mini, compared with their respective unclipped versions. This highlights the importance of tuning the clipping threshold to the noise characteristics, as suggested by the theoretical bound. R-SGD-Mini achieves the best results both without clipping and with clipping at $\lambda=5.$ The other algorithms plateau after approximately $4{,}000$ iterations, whereas R-SGD-Mini and its variants continue to improve over the observed iterations.
 Moreover, it can be observed that the runtime of R-SGD-Mini and R-CSGD-Mini is significantly lower than that of algorithms based on MoM~\cite{puchkinMedian}.

\section{Conclusion} \label{sec:conlusion}

We addressed stochastic first-order optimization when gradient evaluations are subject to heavy-tailed noise with infinite variance. We showed that, for a broad class of symmetric noise distributions, the medoid-of-means gradient estimator yields effective gradient noise with a finite second moment, even when the underlying noise has infinite variance.

For R-SGD-Mini, we established bounds on the expected time-averaged squared gradient norm under standard assumptions on the nonconvex population loss. With a fixed step size, the bound approaches an explicitly quantified limit at a rate of $\mathcal{O}(T^{-1})$. This limit depends on the objective, the noise, and the algorithm parameters. When the time horizon $T$ is known in advance, a horizon-dependent step size yields an $\mathcal{O}(T^{-1/2})$ convergence rate. A diminishing step size achieves the same rate up to a logarithmic factor. When a suitable reference point is available, the same rate can also be achieved under asymmetric noise. Allowing the sample size to grow with the iteration count makes it possible to surpass existing lower bounds and recover the $\mathcal{O}(T^{-1})$ rate of deterministic gradient descent.

The finite second moment allows the estimator to be incorporated into existing algorithms and to inherit their high-probability guarantees. Examples~\ref{ex:vsnguyen} and~\ref{ex:vsnguyen} illustrate how using the proposed estimator in R-CSGD-Mini and R-NSGD-Mini can tighten the resulting bounds. The medoid, the coordinatewise median and geometric medain yield lighter-tailed estimates. The medoid and the geometric median also account for the multivariate geometry of the block means. However, computing the geometric median requires solving an additional optimization problem, whereas the medoid is selected directly from the block means. Experiments on synthetic data and CIFAR-10 demonstrate the robustness of R-SGD-Mini and its variants to heavy-tailed gradient noise.

\appendix
\section{Extension to Asymmetric Noise} \label{app:assymetriextenison}
We replace Assumption~\ref{as:gradientnoise} with the following assumption.
\begin{myassump}\label{ass:extension}
   1) At the start of training we have access to a noiseless gradient of $F$ at any reference point $\mathbf{y}\in\mathbb{R}^d$; 2) For any $\mathbf{x}_{1}, \mathbf{x}_2 \in \mathbb{R}^d $ any two per-point gradient noises $\boldsymbol{\nu}_{\boldsymbol{\xi}_1}(\mathbf{x}_2)$ and $\boldsymbol{\nu}_{\boldsymbol{\xi}_2}(\mathbf{x}_2) $ are independent, identically distributed, and absolutely continuous; 3) There exist constants $\sigma>0$ and $p>1$ such that
$\sup\limits_{\mathbf{x}}\mathbb{E}\left[ \|\boldsymbol{\nu}_{\boldsymbol{\xi}}(\mathbf{x})\|^p \right] \leq \sigma^p.$
\end{myassump}
We also modify Steps 3, 4, and 5 of Algorithm~\ref{alg1} as follows.
\begin{enumerate}
    \item[3.] Sample $K$ data samples $\boldsymbol{\xi}_1, \boldsymbol{\xi}_2,...,\boldsymbol{\xi}_K$ from distribution $\Xi$, and split them into $2M=\frac{K}{R}$ equally-sized data chunks 
    \item[4.] Set that $\nabla f_j^t(\cdot)=\frac{1}{R}\sum\limits_{k=1}^R \nabla f(\cdot,\boldsymbol{\xi}_{j_k})$ for $j=1,2,...,M$, and for $j=M+1,M+2,...,2M,$ $\nabla \tilde{f}_{j-M}^t(\cdot)=\frac{1}{R}\sum\limits_{k=1}^R \nabla f(\cdot,\boldsymbol{\xi}_{j_k})$  and $\mathbf{e}_i^t=\nabla f_i^t(\mathbf{x}^t) - \nabla \tilde{f}_i^t(\mathbf{y}) + F(\mathbf{y})$
    \item[5.] Find index $j^\star$ such that
       $j^\star= \argmin_j \sum\limits_{i=1}^N \| \mathbf{e}_i^t -  \mathbf{e}_i^t\|$ and set $\mathbf{g}^t=\mathbf{e}_{j^\star}^t$
\end{enumerate}
Observe that the underlying gradient noise is symmetrized before applying the medoid selection, since we have that
$\mathbf{e}_i^t=\nabla f_i^t(\mathbf{x}^t)  - \nabla \tilde{f}_i^t(\mathbf{y}) + F(\mathbf{y}) =  F(\mathbf{x}^t) + \boldsymbol{\nu}_j^t(\mathbf{x}^t) - F(\mathbf{y}) - \tilde{\boldsymbol{\nu}}_j^t(\mathbf{y}) + F(\mathbf{y})=F(\mathbf{x}^t) + \boldsymbol{\nu}_j^t(\mathbf{x}^t)- \tilde{\boldsymbol{\nu}}_j^t(\mathbf{y}),$
i.e., the quantity  $ \boldsymbol{\nu}_{j,\mathbf{y}}^t(\mathbf{x}^t):=\boldsymbol{\nu}_j^t(\mathbf{x}^t)- \tilde{\boldsymbol{\nu}}_j^t(\mathbf{y})$ is symmetric since it is the difference of two independent and identically distributed random variables. Consequently, $\boldsymbol{\nu}_{j,\mathbf{y}}^t(\mathbf{x}^t)$ satisfies all the assumptions required in Section~\ref{sec:preliminaries}, and hence Theorem~\ref{theorem:convergencetostationarity} holds.
Thus, the symmetric-noise assumption can be avoided under Assumption~\ref{ass:extension}, at the cost of assuming the existence of a reference point $\mathbf{y}$ where the exact gradient is available, together with the assumption that the per-point gradient noises are i.i.d., regardless of both the state and the sampled data point. 

\section{Extension to Increasing Sample Sizes}\label{app:increasingsamplesizes}
Recall that we here consider a variant of algorithm~\ref{alg1} in which the sample size increases with the iteration index $t$. Specifically, we use $K_t = M R_t$ samples at iteration $t$, where $R_t = (t+1)^n$. Recall that we refer to this variant as R-iSGD-Mini.

\begin{theorem}\label{theorem:increasingsamplesizes}
      Suppose that Assumption~\ref{as:Lipshitzgradient} holds and that all conditions in Assumption~\ref{as:gradientnoise} other than symmetry are satisfied. For $\alpha \leq \frac{1}{L}$, $\gamma\in(0,\frac{1}{2}),$ $\theta\in(0,1)$ and let $M>\frac{2}{p\theta\gamma},$ R-iSGD-Mini generates a sequence of iterates $\{\mathbf{x}^t\}$ such that
the average $\frac{1}{T}\sum_{t=0}^{T-1}
\mathbb{E}[\|\nabla F(\mathbf{x}^t)\|^2]$ is
$\mathcal{O}(T^{-2n(p-1)/p})$ if $n < \frac{p}{2(p-1)}$,
$\mathcal{O}((\log T)/T)$ if $n = \frac{p}{2(p-1)}$, and
$\mathcal{O}(1/T)$ if $n > \frac{p}{2(p-1)}$.
\end{theorem}
\begin{proof}
        We have that
\begin{align*}
            F(\mathbf{x}^t-\alpha\nabla f_j(\mathbf{x}^t))&\leq F(\mathbf{x}^t) - \alpha \nabla F(\mathbf{x}^t)^\top (\nabla F(\mathbf{x}^t)+\boldsymbol{\nu}_j^t(\mathbf{x}^t))\\&+\frac{L\alpha^2}{2}\|\nabla F(\mathbf{x}^t)+\boldsymbol{\nu}_j^t(\mathbf{x}^t)\|^2\\
            &=F(\mathbf{x}^t)- \alpha\|\nabla F(\mathbf{x}^t)\|^2 -\alpha \nabla F(\mathbf{x}^t)^\top \boldsymbol{\nu}_j^t (\mathbf{x}^t) \\ 
            &+\frac{L\alpha^2}{2}\|\nabla F(\mathbf{x}^t)+\boldsymbol{\nu}_j^t(\mathbf{x}^t)\|^2\\
            &=F(\mathbf{x}^t)- \alpha\|\nabla F(\mathbf{x}^t)\|^2 \\&-\frac{\alpha}{2} \left(\|\nabla F(\mathbf{x}^t)+\boldsymbol{\nu}_j^t(\mathbf{x}^t)\|^2  - \|\nabla F(\mathbf{x}^t)\|^2 -\|\boldsymbol{\nu}_j^t(\mathbf{x}^t)\|^2 \right)\\ 
            &+\frac{L\alpha^2}{2}\|\nabla F(\mathbf{x}^t)+\boldsymbol{\nu}_j^t(\mathbf{x}^t)\|^2\\
            &= F(\mathbf{x}^t)-\frac{\alpha}{2} \|\nabla F(\mathbf{x}^t)\|^2 + \frac{\alpha}{2} \|\boldsymbol{\nu}_j^t(\mathbf{x}^t)\|^2\\
            &-\frac{\alpha}{2}(1-\alpha L) \|\nabla F(\mathbf{x}^t)+\boldsymbol{\nu}_j^t(\mathbf{x}^t)\|^2 \\
            &\leq F(\mathbf{x}^t)-\frac{\alpha}{2} \|\nabla F(\mathbf{x}^t)\|^2 + \frac{\alpha}{2} \|\boldsymbol{\nu}_j^t(\mathbf{x}^t)\|^2
        \end{align*}
where in the last inequality we used $1-\alpha L\geq 0.$ Thus,
            $\mathbb{E}[F(\mathbf{x}^{t+1})]\leq \mathbb{E}[F(\mathbf{x}^t)]-\frac{\alpha}{2} \mathbb{E}[\|\nabla F(\mathbf{x}^t)\|^2] +  \frac{\alpha}{2} B_t,$
        where we recall $B_t$ from~\eqref{eqn:truesecondmoment}, i.e.,
\begin{align*}
    B_t
    &=\left(\frac{\sigma((3-2\gamma)M+3)}{(t+1)^{\frac{n(p-1)}{p}}((1-2\gamma)M) \left((1-\gamma)^{1-\gamma}\gamma^\gamma\right)^{\frac{1}{(1-\theta)p\gamma}} }\right)^2 \left(1+\frac{2}{p\theta\gamma M -2} \right)\\
    &=\mathcal{O}\left((t+1)^{-\frac{2n(p-1)}{p}}\right).
\end{align*}
Hence,$\frac{\alpha}{2}\sum\limits_{t=0}^{T-1}\mathbb{E}[\|\nabla F(\mathbf{x}^t)\|^2]\leq F(\mathbf{x}^0)- F^\star + \frac{\alpha}{2} B_t$
and
\begin{align*}
    \frac{1}{T}\sum\limits_{t=0}^{T-1}\mathbb{E}[\|\nabla F(\mathbf{x}^t)\|^2]&\leq \frac{2\left(F(\mathbf{x}^0)-F^\star\right) }{\alpha T}+\frac{1}{T}\sum\limits_{t=0}^{T-1}B_t.
\end{align*}
Since the average $\frac{1}{T}\sum_{t=0}^{T-1} B_t$ is
$\mathcal{O}(T^{-2n(p-1)/p})$ if $n < \frac{p}{2(p-1)}$,
$\mathcal{O}((\log T)/T)$ if $n = \frac{p}{2(p-1)}$, and
$\mathcal{O}(1/T)$ if $n > \frac{p}{2(p-1)}$
this concludes the proof.
\end{proof}

\section{Extension to diminishing step size} \label{app:diminishingextension}
If Step 6 of Algorithm~\ref{alg1} is replaced with $\mathbf{x}^{t+1}=\mathbf{x}^t-\alpha_t \mathbf{g}^t,$ for $\alpha_t = \frac{a}{(t+1)^{\frac{1}{2}}}, a>0$, the following theorem holds.
\begin{theorem}\label{theorem:convergencetostationaritydiminishin}
        Let Assumptions~\ref{as:Lipshitzgradient} and~\ref{as:gradientnoise} hold. For $a \leq \frac{1}{2L}$, $\gamma\in(0,\frac{1}{2}),$ $\theta\in(0,1)$ and let $M>\frac{2}{p\theta\gamma},$ Algorithm~\ref{alg1} generates the sequence of iterates $\{\mathbf{x}^t\}$ such that
        \begin{align*}
	\min\limits_{0\leq k\leq t-1} \mathbb{E}[\|\nabla F(\mathbf{x}^k)\|^2] \leq \frac{
    F(\mathbf{x}^0)- F^\star+LBa^2(1+\log t)}{a(\sqrt{t+1}-1)},
\end{align*}
where $B=\left(\frac{\sigma((3-2\gamma)M+3)}{R^{\frac{p-1}{p}}((1-2\gamma)M) \left((1-\gamma)^{1-\gamma}\gamma^\gamma\right)^{\frac{1}{(1-\theta)p\gamma}} }\right)^2 \left(1+\frac{2}{p\theta\gamma M -2} \right).$
    \end{theorem}
    \begin{proof}
       In the proof, we highlight only the differences from the proof of Theorem~\ref{theorem:convergencetostationarity}.
        Recalling that $\mathbb{E}_t[\cdot]=\mathbb{E}[\cdot|\mathcal{F}_t]$  and utilizing Theorem~\ref{theo:variance} and Lemma~\ref{lemma:zeromean} we have that
         \begin{align}
	\mathbb{E}_t\left[  F(\mathbf{x}^t-\alpha_t\nabla f_{j^\star}^t(\mathbf{x}^t))\right]& \leq F(\mathbf{x}^t) - \frac{\alpha_t}{2}\|\nabla F(\mathbf{x}^t)\|^2 -\alpha_t \nabla F(\mathbf{x}^t)^\top \mathbb{E}_t\left[\boldsymbol{\nu}_{j^\star}^t\right]\nonumber\\& + L\alpha_t^2\mathbb{E}_t\left[ \|\boldsymbol{\nu}_{j^\star}^t\|^2\right]\nonumber\\
	&\leq F(\mathbf{x}^t) - \frac{\alpha_t}{2}\|\nabla F(\mathbf{x}^t)\|^2+ L\alpha_t^2 B,\label{eqn:prefinald}
\end{align}
where $B=\left(\frac{\sigma((3-2\gamma)M+3)}{R^{\frac{p-1}{p}}((1-2\gamma)M) \left((1-\gamma)^{1-\gamma}\gamma^\gamma\right)^{\frac{1}{(1-\theta)p\gamma}} }\right)^2 \left(1+\frac{2}{p\theta\gamma M -2} \right)$ is given by~\eqref{eqn:truesecondmoment}.
Finally, by taking the expectation in~\eqref{eqn:prefinald}, we have the following
\begin{align*}
	\mathbb{E}[F(\mathbf{x}^{t+1})]&\leq \mathbb{E}[F(\mathbf{x}^t)]-\frac{\alpha_t}{2} \mathbb{E}[\|\nabla F(\mathbf{x}^t)\|^2] +  L\alpha_t^2 B,
\end{align*}
hence, $\frac{1}{2}\sum\limits_{k=0}^{t-1}\alpha_k\mathbb{E}[\|\nabla F(\mathbf{x}^k)\|^2]\leq F(\mathbf{x}^0)- F^\star +  L B \sum\limits_{k=0}^{t-1}\alpha_k^2.$ 

Since $\min\limits_{0\leq k\leq t-1} \mathbb{E}[\|\nabla F(\mathbf{x}^k)\|^2]\leq \mathbb{E}[\|\nabla F(\mathbf{x}^k)\|^2]$ for $k=0,1,...,t-1,$ we get
\begin{align*}
	\min\limits_{0\leq k\leq t-1} \mathbb{E}[\|\nabla F(\mathbf{x}^k)\|^2]&\leq \frac{2(F(\mathbf{x}^0)- F^\star)+2L B \sum\limits_{k=0}^{t-1}\alpha_k^2}{\sum\limits_{k=0}^{t-1}\alpha_k}, 
\end{align*}
and since $\sum\limits_{k=0}^{t-1}\frac{1}{k+1}\leq 1+\log t$ and $\sum\limits_{k=0}^{t-1}\frac{1}{(k+1)^{\frac{1}{2}}}\geq 2(\sqrt{t+1}-1),$ we have 
\begin{align*}
	\min\limits_{0\leq k\leq t-1} \mathbb{E}[\|\nabla F(\mathbf{x}^k)\|^2] \leq \frac{
    F(\mathbf{x}^0)- F^\star+LBa^2(1+\log t)}{a(\sqrt{t+1}-1)},
\end{align*}
which completes the proof.
    \end{proof}

\section{Theorem 5.1 in \cite{nguyen2023improved}} \label{app:theoHPB}
Theorem 5.1 of~\cite{nguyen2023improved} reads as follows.
\begin{theorem}\label{theo:nguyen}
    Let Assumptions~\ref{as:Lipshitzgradient} and~\ref{as:gradientnoise} hold. Let $\tau=\max\{\log\frac{1}{\delta},1\}$ and $\Delta_1=F(\mathbf{x}_1)-F^\star.$ For known $T,$ choose $\lambda_t$ and $\eta_t$ such that
    \begin{align*}
        \lambda_t&=\lambda=\max \Bigg\{\left(\frac{8\tau}{\sqrt{L\Delta_1}} \right)^{\frac{1}{p-1}}T^{\frac{1}{3p-2}}\sigma^{\frac{p}{p-1}}, 2\sqrt{90L\Delta_1},32^{\frac{1}{p}}\sigma T^{\frac{1}{3p-2}} \Bigg\},\\
        \alpha_t&=\alpha=\frac{\sqrt{\Delta_1}T^{\frac{1-p}{3p-2}}}{8\lambda\sqrt{L}\tau}=\frac{\sqrt{\Delta_1}}{8\sqrt{L}\tau}\min \Bigg\{\left(\frac{8\tau}{\sqrt{L\Delta_1}}\right)^{\frac{-1}{p-1}}T^{\frac{-p}{3p-2}}\sigma^{\frac{-p}{p-1}},\frac{T^{\frac{1-p}{3p-2}}}{2\sqrt{90L\Delta_1}},\frac{T^{\frac{-p}{3p-2}}}{32^{\frac{1}{p}}\sigma}\Bigg\}.
    \end{align*}
    Then with probability at least $1-\delta$
    \begin{align*}
        \frac{1}{T}\sum\limits_{t=1}^T&\|\nabla F(\mathbf{x^t}\|^2\\&\leq 720\sqrt{\Delta_1 L}\tau \max \Bigg\{ \left(\frac{8\tau}{\sqrt{L\Delta_1}} \right)^{\frac{1}{p-1}}T^{\frac{2-2p}{3p-2}}\sigma^{\frac{p}{p-1}},2\sqrt{90L\Delta_1}T^{\frac{1-2p}{3p-2}},32^{\frac{1}{p}} \sigma T^{\frac{2-2p}{3p-2}} \Bigg\}\\&=\mathcal{O}(T^{\frac{2-2p}{3p-2}}).
    \end{align*}
    \end{theorem}
    By Lemma~\ref{lemma:zeromean} and Theorem~\ref{theo:variance}, for all sufficiently large $M$, the medoid gradient estimator satisfies the assumptions of Theorem~\ref{theo:nguyen}, with its effective central moment bounded by
    \begin{align*}
        \sigma_{j^\star}^2=\mathbb{E}\left[\|\boldsymbol{\nu}_{j^\star}\|^2 \right] \leq \left(\frac{\sigma((3-2\gamma)M+3)}{R^{\frac{p-1}{p}}((1-2\gamma)M) \left((1-\gamma)^{1-\gamma}\gamma^\gamma\right)^{\frac{1}{(1-\theta)p\gamma}} }\right)^2 \left(1+\frac{2}{p\theta\gamma M -2} \right).
    \end{align*}
    Thus, for the parameter choices specified in Theorem~\ref{theo:nguyen} and all sufficiently large $M$, R-CSGD-Mini generates a sequence of iterates $\{\mathbf{x}^t\}$ such that
\begin{align*}
        &\frac{1}{T}\sum\limits_{t=1}^T\|\nabla F(\mathbf{x^t}\|^2\\&\leq 720\sqrt{\Delta_1 L}\tau \max \Bigg\{ \left(\frac{8\tau}{\sqrt{L\Delta_1}} \right)T^{-\frac{1}{2}}\sigma_{j^\star}^{2},2\sqrt{90L\Delta_1}T^{-\frac{3}{4}},\sqrt{32} \sigma_{j^\star} T^{-\frac{1}{2}} \Bigg\}=\mathcal{O}(T^{-\frac{1}{2}}).
    \end{align*}
\section{Assumptions and Theorem 1 in \cite{armacki}} \label{app:armacki}
We first state the assumption on the gradient noise that allows us to apply the results of~\cite{armacki}.
\begin{myassump}\label{as:gradientnoisarm}
    Suppose that Assumption~\ref{as:gradientnoise} holds and that, conditional on $\mathbf{x}^t=\mathbf{x}$, the gradient noise $\boldsymbol{\nu}_{\boldsymbol{\xi}}(\mathbf{x})$ admits a probability density $P_{\mathbf{x}}\colon\mathbb{R}^d\to[0,\infty)$. Assume further that these conditional densities are uniformly bounded away from zero in a neighborhood of the origin, i.e., $\inf_{\mathbf{x} \in \mathbb{R}^d}P_{\mathbf{z}}(\mathbf{z}) > 0$, for all $\|\mathbf{z}\| \leq E_0$ and some $E_0 > 0$.
\end{myassump}

The next two assumptions concern the nonlinearity applied to the gradient estimator.
\begin{myassump}\label{asmpt:nonlin}
The nonlinear map $\boldsymbol{\Psi}: \mathbb{R}^d \mapsto \mathbb{R}^d$ is either component-wise, i.e., of the form $\boldsymbol{\Psi}(\mathbf{x}) = \left[ \psi(x_1), \dots, \psi(x_d) \right]^\top$, or joint, i.e., of the form $\boldsymbol{\Psi}(\mathbf{x}) = \mathbf{x}\varphi(\|\mathbf{x}\|)$, where the mappings $\psi,\:\varphi: \mathbb{R} \mapsto \mathbb{R}$ satisfy
1) $\psi,\varphi$ are continuous, except for at most finitely many points, with $\psi$ piece-wise differentiable and $|\psi^\prime(x)| \leq c_1$, for some $c_1 > 0$, while $a \mapsto a\varphi(a)$ is non-decreasing on $a \in (0,\infty)$.
2) $\psi$ is monotonically non-decreasing and odd, while $\varphi$ is non-increasing on $(0,\infty)$.
3) $\psi$ is either discontinuous at zero, or strictly increasing on $(-c_2,c_2)$, for some $c_2 > 0$, with $\varphi(a) > 0$, for any $a > 0$.
4) $\psi$ and $\mathbf{x}\varphi(\|\mathbf{x}\|)$ are uniformly bounded, i.e., for some $C_1,\: C_2 > 0$, and all $x \in \mathbb{R}$, $\mathbf{x} \in \mathbb{R}^d$, we have $|\psi(x)| \leq C_1$ and $\|\mathbf{x}\varphi(\|\mathbf{x}\|)\|\leq C_2$.
\end{myassump}

\begin{myassump}\label{asmpt:nonlin-state}
    If $\boldsymbol{\Psi}: \mathbb{R}^d \mapsto \mathbb{R}^d$ is joint, i.e., of the form $\boldsymbol{\Psi}(\mathbf{x}) = \mathbf{x}\varphi(\|\mathbf{x}\|)$, then the conditions of Assumption \ref{asmpt:nonlin} hold. If $\boldsymbol{\Psi}$ is component-wise, i.e., of the form $\boldsymbol{\Psi}(\mathbf{x}) = \left[ \psi(x_1), \dots, \psi(x_d) \right]^\top$, then in addition to conditions from Assumption \ref{asmpt:nonlin}, $\psi$ is twice continuously differentiable, with uniformly bounded first and second derivatives, i.e., $|\psi^\prime(x)|\leq K_1$ and $|\psi^{\prime\prime}(x)|\leq K_2$, for some $K_1,K_2 > 0$.
\end{myassump}
Before restating Theorem 1 from~\cite{armacki}, we restate the following lemma (Lemma 2 from~\cite{armacki}).

\begin{lemma}\label{lm:key-unified-state}
    Let Assumptions \ref{as:gradientnoisarm} and \ref{asmpt:nonlin-state} hold. Then, for any $\mathbf{x,y} \in \mathbb{R}^d$, we have $\langle \boldsymbol{\Phi}_{\mathbf{x}}(\mathbf{y}),\mathbf{y}\rangle \geq \min\left\{\beta_1\|\mathbf{y}\|,\beta_2\|\mathbf{y}\|^2 \right\}$, where $\beta_1,\beta_2 > 0$, are constants that only depend on noise, choice of nonlinearity and other problem related parameters and $\boldsymbol{\Phi}_{\mathbf{x}}(\mathbf{y}) =\mathbb{E}_{\mathbf{x}} [\boldsymbol{\Psi} (\mathbf{y} + \mathbf{z})].$
\end{lemma}

Finally, we present a modified version of Theorem 1 in~\cite{armacki}. All its other assumptions are already satisfied. It remains to verify that the density of the medoid noise is uniformly bounded away from zero in a neighborhood of the origin.
\begin{theorem}\label{thm:n-sgd}
    Let Assumptions \ref{as:Lipshitzgradient}, \ref{as:gradientnoisarm} and \ref{asmpt:nonlin-state} hold and let $\{ \mathbf{x}^t\}_{t \in \mathbb{N}}$ be the sequence of models generated by Algorithm~\ref{alg1}, with step-size schedule $\alpha_t = \frac{a}{(t+1)^{\eta}}$, for any $\eta \in [1/2,1)$. If the step-size parameter satisfies $a \leq \min\left\{ 1,\frac{\beta_2}{8C^2}, \frac{(1-\eta)\beta_1}{8C^2(LC+(1-\eta)\widetilde{\nabla})}\right\}$, then for any $t \in \mathbb{N}$ and any $\delta \in (0,1)$, with probability at least $1 - \delta$, we have 
    \begin{equation*}
        \frac{1}{t}\sum_{k = 1}^t\min\{\|\nabla F(\mathbf{x}^k)\|,\|\nabla F(\mathbf{x}^k)\|^2\} \leq \frac{4(\log(\nicefrac{1}{\delta}) + \Delta)}{a\beta t^{1-\eta}} + \frac{2aLC^2}{\beta t^{1-\eta}}\sum_{k = 1}^t(k+1)^{-2\eta},
    \end{equation*}
    where $\beta = \min\{\beta_1,\beta_2\}$ and $ \widetilde{\nabla} = \|\nabla F(\mathbf{x}^1)\|$
\end{theorem}
\begin{proof} 
It suffices to show that the density of the medoid noise is uniformly bounded away from zero in a neighborhood of the origin.
For simplicity of notation, we omit $\boldsymbol{\xi}$ and $\mathbf{x}$ from the notation for the noise. Instead, we use $p_{\mathbf{x}}$ to denote the conditional probability density function given $\mathbf{x}^t=\mathbf{x}$, and $P_{\mathbf{x}}$ for the corresponding conditional probability
$P_{\mathbf{x}}\{A\}=P\{A\mid \mathbf{x}^t=\mathbf{x}\}$.
Let us assume that there exists $E_0>0$ such that $\forall \mathbf{u}\in\mathbb{R}^d, \|\mathbf{u}\|\leq E_0$ we have that $\inf_{\mathbf{x}\in\mathbb{R}^d}  p_{\mathbf{x}}(\mathbf{u})>0.$ 
Moreover, let $\boldsymbol{\nu}_1, \boldsymbol{\nu}_2,...,\boldsymbol{\nu}_M$ be independent and identically distributed random variables with probability density function $p_{\mathbf{x}}$ and suppose that $\boldsymbol{\nu}_{j^\star}$ is their medoid and $p_{\mathbf{x}}^{\star}(\mathbf{u})$ its probability density function for fixed $\mathbf{x}\in\mathbb{R}^d$. It is suffice to show that there exists $E_1>0$ such that $\forall \mathbf{u}\in\mathbb{R}^d, \|\mathbf{u}\|\leq E_1$ we have that $\inf_{\mathbf{x}\in\mathbb{R}^d}  p_{\mathbf{x}}^{\star}(\mathbf{u})>0.$

We have that $P_{\mathbf{x}}\{\boldsymbol{\nu}_{j^\star}\in A  \}= \sum\limits_{i=1}^M P_{\mathbf{x}}\{ \boldsymbol{\nu}_i \in A, j^{\star}=i \} = M  P_{\mathbf{x}}\{ \boldsymbol{\nu}_1 \in A, j^{\star}=1 \}=M\int\limits_A p_{\mathbf{x}}(\mathbf{u}) P_{\mathbf{x}}\{ j^\star =1 | \boldsymbol{\nu}_1 =\mathbf{u} \}d\mathbf{u}$ and therefore, probability density function of medoid $\boldsymbol{\nu}_{j^\star}$  is 
    $p_{\mathbf{x}}^{\star}(\mathbf{u})= Mp_{\mathbf{x}}(\mathbf{u}) P_{\mathbf{x}}\{ j^\star =1 | \boldsymbol{\nu}_1 =\mathbf{u} \}.$
Note that $p_{\mathbf{x}}(\mathbf{u})>0$ follows from the assumption. We now show that $P_{\mathbf{x}}\{ j^\star =1 | \boldsymbol{\nu}_1 =\mathbf{u} \}>0$ holds.

Let $\mathbf{u}$ is such that $\|\mathbf{u}\| \leq E_1 = \frac{E_0}{2}.$  If we consider points
$\mathbf{y}_i=\mathbf{u}+\mathbf{v}_i,$ for $i=2,3,...,M,$
where $\mathbf{v}_2,\mathbf{v}_3,...,\mathbf{v}_M$ are the vertices of a regular \((M-1)\)-gon centered at the origin and $\|\mathbf{v}_i\| = \rho$ and $\rho$ is small enough such that $\|\mathbf{y}_i\|\leq E_0,$ then $\mathbf{u}$ is the medoid of points $\mathbf{u},\mathbf{v}_2,\mathbf{v}_3,...,\mathbf{v}_M$
Moreover, since the distance functions $D(\cdot)$  are continuous, there exists a sufficiently small $\varepsilon>0$ such that $\mathbf{u}$ remains the medoid when each point $\mathbf{y}_i$ is perturbed slightly.
More precisely, if
$\boldsymbol{\nu}_i \in B(\mathbf{y}_i,\varepsilon),i=2,3,...,M$
then $\boldsymbol{\nu}_1=\mathbf{u}$ is the medoid of the set $\mathbf{u},\boldsymbol{\nu}_2, \boldsymbol{\nu}_3,...,\boldsymbol{\nu}_N.$
We choose $\varepsilon$ sufficiently small so that
$
B(\mathbf{y}_i,\varepsilon)\subset B(0,E_0)$
for every $i=2,3,...,M$ also holds.

Define the event $A_{\mathbf{u}}=
\bigcap_{i=2}^{M}
\{\boldsymbol{\nu}_i \in B(\mathbf{y}_i,\varepsilon)\}.$  
On the event $A_{\mathbf{u}}$, the point $\boldsymbol{\nu}_1 =\mathbf{u}$ is the medoid. Therefore,$
P_{\mathbf{x}}\{ j^\star =1 | \boldsymbol{\nu}_1 =\mathbf{u} \}\geq   P_{\mathbf{x}}\{A_{\mathbf{u}} | \boldsymbol{\nu}_1 =\mathbf{u}\}= \prod_{i=2}^{M}
P_{\mathbf{x}}\{ \boldsymbol{\nu} \in B(\mathbf{y}_i,\varepsilon)\},$
where the last inequality holds since the points are i.i.d., and $\boldsymbol{\nu}_2,\boldsymbol{\nu},...,\boldsymbol{\nu}_M$ are independent of $\boldsymbol{\nu}_1$.

Suppose that $\inf_{\mathbf{x}\in\mathbb{R}^d} P_{\mathbf{x}}\{ \boldsymbol{\nu} \in B(\mathbf{y}_i,\varepsilon)\}=0,$ then we have that there exits a sequence $\{\mathbf{x}_n\}$ such that 
$\lim\limits_{n\to\infty} P_{\mathbf{x}_n}\{ \boldsymbol{\nu} \in B(\mathbf{y}_i,\varepsilon)\} =0.$ Utilizing Fatou's lemma, we obtain a contradiction because $0 = \liminf\limits_{n\to\infty} P_{\mathbf{x}_n}\{ \boldsymbol{\nu} \in B(\mathbf{y}_i,\varepsilon)\}  =  \liminf\limits_{n\to\infty} \int_{B(\mathbf{y}_i,\varepsilon)} p_{\mathbf{x}_n}(\mathbf{u})d\mathbf{u} \geq  \int_{B(\mathbf{y}_i,\varepsilon)}  \liminf\limits_{n\to\infty} p_{\mathbf{x}_n}(\mathbf{u})d\mathbf{u} >0,$
since $  \liminf\limits_{n\to\infty} p_{\mathbf{x}_n}(\mathbf{u}) \geq \inf_{\mathbf{x}\in\mathbb{R}^d}  p_{\mathbf{x}}(\mathbf{u})>0.$
Finally, we have that
    $\inf_{\mathbf{x}\in\mathbb{R}^d} p_{\mathbf{x}}^{\star}(\mathbf{u}) \geq  M \inf_{\mathbf{x}\in\mathbb{R}^d} p_{\mathbf{x}}(\mathbf{u}) \inf_{\mathbf{x}\in\mathbb{R}^d} P_{\mathbf{x}}\{ j^\star =1 | \boldsymbol{\nu}_1 =\mathbf{u} \}>0,$
$\forall \mathbf{u}\in\mathbb{R}^d, \|\mathbf{u}\|\leq E_1$
\end{proof}
Note that with a fixed step size $\alpha=\frac{a}{\sqrt{T}}$, Theorem~\ref{thm:n-sgd} yields the following bound.
\begin{align}\label{eqn:armbound1}
     \frac{1}{T}\sum_{k = 0}^{T-1}\min\{\|\nabla F(\mathbf{x}^k)\|,\|\nabla F(\mathbf{x}^k)\|^2\} 
     &=\frac{4(\log(\nicefrac{1}{\delta}) + \Delta)}{a\beta \sqrt{T}} + \frac{2a LC^2}{\beta\sqrt{T}}
\end{align}

\bibliographystyle{siam}
\bibliography{references}

@article{CSGDHT,
  title={High probability convergence of clipped-sgd under heavy-tailed noise},
  author={Nguyen, Ta Duy and Nguyen, Thien Hang and Ene, Alina and Nguyen, Huy Le},
  journal={arXiv preprint arXiv:2302.05437},
  year={2023}
}

@article{HTASGD,
  title={Stochastic optimization with heavy-tailed noise via accelerated gradient clipping},
  author={Gorbunov, Eduard and Danilova, Marina and Gasnikov, Alexander},
  journal={Advances in Neural Information Processing Systems},
  volume={33},
  pages={15042--15053},
  year={2020}
}

@article{robinsmonro,
  title={A stochastic approximation method},
  author={Robbins, Herbert and Monro, Sutton},
  journal={The annals of mathematical statistics},
  pages={400--407},
  year={1951},
  publisher={JSTOR}
}

@article{ghadimi,
  title={Stochastic first-and zeroth-order methods for nonconvex stochastic programming},
  author={Ghadimi, Saeed and Lan, Guanghui},
  journal={SIAM journal on optimization},
  volume={23},
  number={4},
  pages={2341--2368},
  year={2013},
  publisher={SIAM}
}

@article{nemirovski,
  title={Robust stochastic approximation approach to stochastic programming},
  author={Nemirovski, Arkadi and Juditsky, Anatoli and Lan, Guanghui and Shapiro, Alexander},
  journal={SIAM Journal on optimization},
  volume={19},
  number={4},
  pages={1574--1609},
  year={2009},
  publisher={SIAM}
}

@article{zhangHT,
  title={Why are adaptive methods good for attention models?},
  author={Zhang, Jingzhao and Karimireddy, Sai Praneeth and Veit, Andreas and Kim, Seungyeon and Reddi, Sashank and Kumar, Sanjiv and Sra, Suvrit},
  journal={Advances in Neural Information Processing Systems},
  volume={33},
  pages={15383--15393},
  year={2020}
}

@article{nairHT,
  title={The Fundamentals of Heavy Tails},
  author={Nair, Jayakrishnan and Wierman, Adam and Zwart, Bert},
  year={2013}
}

@misc{armacki,
      title={Sharp High-Probability Rates for Nonlinear SGD under Heavy-Tailed Noise via Symmetrization}, 
      author={Aleksandar Armacki and Dragana Bajovic and Dusan Jakovetic and Soummya Kar},
      year={2026},
      eprint={2507.09093},
      archivePrefix={arXiv},
      primaryClass={stat.ML},
      url={https://arxiv.org/abs/2507.09093}, 
}

@article{dusan,
author = {Jakovetic, Dusan and Bajovic, Dragana and Sahu, Anit Kumar and Kar, Soummya and Milosevic, Nemanja and Stamenkovic, Dusan},
title = {Nonlinear Gradient Mappings and Stochastic Optimization: A General Framework with Applications to Heavy-Tail Noise},
journal = {SIAM Journal on Optimization},
volume = {33},
number = {2},
pages = {394-423},
year = {2023},
doi = {10.1137/21M145896X},

URL = { 
    
        https://doi.org/10.1137/21M145896X
    
    

},
eprint = { 
    
        https://doi.org/10.1137/21M145896X
    
    

}
}

@inproceedings{puchkinMedian,
  title={Breaking the heavy-tailed noise barrier in stochastic optimization problems},
  author={Puchkin, Nikita and Gorbunov, Eduard and Kutuzov, Nickolay and Gasnikov, Alexander},
  booktitle={International Conference on Artificial Intelligence and Statistics},
  pages={856--864},
  year={2024},
  organization={PMLR}
}

@article{cutkosky2021high,
  title={High-probability bounds for non-convex stochastic optimization with heavy tails},
  author={Cutkosky, Ashok and Mehta, Harsh},
  journal={Advances in Neural Information Processing Systems},
  volume={34},
  pages={4883--4895},
  year={2021}
}

@inproceedings{liu2023breaking,
  title={Breaking the lower bound with (little) structure: Acceleration in non-convex stochastic optimization with heavy-tailed noise},
  author={Liu, Zijian and Zhang, Jiawei and Zhou, Zhengyuan},
  booktitle={The Thirty Sixth Annual Conference on Learning Theory},
  pages={2266--2290},
  year={2023},
  organization={PMLR}
}

@article{nguyen2023improved,
  title={Improved convergence in high probability of clipped gradient methods with heavy tailed noise},
  author={Nguyen, Ta Duy and Nguyen, Thien H and Ene, Alina and Nguyen, Huy},
  journal={Advances in Neural Information Processing Systems},
  volume={36},
  pages={24191--24222},
  year={2023}
}

@article{minsker2015geometric,
  title={Geometric median and robust estimation in Banach spaces},
  author={Minsker, Stanislav},
  year={2015}
}

@article{lugosi2019mean,
  title={Mean estimation and regression under heavy-tailed distributions: A survey},
  author={Lugosi, G{\'a}bor and Mendelson, Shahar},
  journal={Foundations of Computational Mathematics},
  volume={19},
  number={5},
  pages={1145--1190},
  year={2019},
  publisher={Springer}
}

@article{lugosi2019sub,
  title={Sub-Gaussian estimators of the mean of a random vector},
  author={Lugosi, G{\'a}bor and Mendelson, Shahar},
  year={2019}
}

@inproceedings{schubert2019faster,
  title={Faster k-medoids clustering: improving the PAM, CLARA, and CLARANS algorithms},
  author={Schubert, Erich and Rousseeuw, Peter J},
  booktitle={International conference on similarity search and applications},
  pages={171--187},
  year={2019},
  organization={Springer}
}

@article{schubert2021fast,
  title={Fast and eager k-medoids clustering: O (k) runtime improvement of the PAM, CLARA, and CLARANS algorithms},
  author={Schubert, Erich and Rousseeuw, Peter J},
  journal={Information Systems},
  volume={101},
  pages={101804},
  year={2021},
  publisher={Elsevier}
}

@article{chen2020understanding,
  title={Understanding gradient clipping in private sgd: A geometric perspective},
  author={Chen, Xiangyi and Wu, Steven Z and Hong, Mingyi},
  journal={Advances in neural information processing systems},
  volume={33},
  pages={13773--13782},
  year={2020}
}

@inproceedings{bernstein2018signsgd,
  title={signSGD: Compressed optimisation for non-convex problems},
  author={Bernstein, Jeremy and Wang, Yu-Xiang and Azizzadenesheli, Kamyar and Anandkumar, Animashree},
  booktitle={International conference on machine learning},
  pages={560--569},
  year={2018},
  organization={PMLR}
}

@article{barsbey2021heavy,
  title={Heavy tails in SGD and compressibility of overparametrized neural networks},
  author={Barsbey, Melih and Sefidgaran, Milad and Erdogdu, Murat A and Richard, Gael and Simsekli, Umut},
  journal={Advances in neural information processing systems},
  volume={34},
  pages={29364--29378},
  year={2021}
}

@inproceedings{battash2024revisiting,
  title={Revisiting the noise model of stochastic gradient descent},
  author={Battash, Barak and Wolf, Lior and Lindenbaum, Ofir},
  booktitle={International Conference on Artificial Intelligence and Statistics},
  pages={4780--4788},
  year={2024},
  organization={PMLR}
}

@inproceedings{armacki2025high,
  title={High-probability convergence bounds for online nonlinear stochastic gradient descent under heavy-tailed noise},
  author={Armacki, Aleksandar and Yu, Shuhua and Sharma, Pranay and Joshi, Gauri and Bajovic, Dragana and Jakovetic, Dusan and Kar, Soummya},
  booktitle={The 28th International Conference on Artificial Intelligence and Statistics},
  volume={258},
  pages={9},
  year={2025}
}

@article{bernstein2018signsgd2,
  title={signSGD with majority vote is communication efficient and fault tolerant},
  author={Bernstein, Jeremy and Zhao, Jiawei and Azizzadenesheli, Kamyar and Anandkumar, Anima},
  journal={arXiv preprint arXiv:1810.05291},
  year={2018}
}

@inproceedings{gurbuzbalaban2021heavy,
  title={The heavy-tail phenomenon in SGD},
  author={Gurbuzbalaban, Mert and Simsekli, Umut and Zhu, Lingjiong},
  booktitle={International Conference on Machine Learning},
  pages={3964--3975},
  year={2021},
  organization={PMLR}
}

@inproceedings{simsekliHT,
  title={A tail-index analysis of stochastic gradient noise in deep neural networks},
  author={Simsekli, Umut and Sagun, Levent and Gurbuzbalaban, Mert},
  booktitle={International Conference on Machine Learning},
  pages={5827--5837},
  year={2019},
  organization={PMLR}
}

@inproceedings{peluchetti2020stable,
  title={Stable behaviour of infinitely wide deep neural networks},
  author={Peluchetti, Stefano and Favaro, Stefano and Fortini, Sandra},
  booktitle={International Conference on Artificial Intelligence and Statistics},
  pages={1137--1146},
  year={2020},
  organization={PMLR}
}

@inproceedings{resnet,
  title={Deep residual learning for image recognition},
  author={He, Kaiming and Zhang, Xiangyu and Ren, Shaoqing and Sun, Jian},
  booktitle={Proceedings of the IEEE conference on computer vision and pattern recognition},
  pages={770--778},
  year={2016}
}

@article{cifar10,
  title={Learning multiple layers of features from tiny images},
  author={Krizhevsky, Alex and Hinton, Geoffrey and others},
  year={2009},
  publisher={Toronto, ON, Canada}
}

@article{armackiocekivanje,
author = {Armacki, Aleksandar and Yu, Shuhua and Bajovi{\'c}, Dragana and Jakoveti{\'c}, Du{\v s}an and Kar, Soummya},
title = {Large Deviation Upper Bounds and Improved MSE Rates of Nonlinear SGD: Heavy-Tailed Noise and Power of Symmetry},
journal = {SIAM Journal on Optimization},
volume = {36},
number = {1},
pages = {32-59},
year = {2026},
doi = {10.1137/24M1704154},

URL = { 
    
        https://doi.org/10.1137/24M1704154
    
    

},
eprint = { 
    
        https://doi.org/10.1137/24M1704154
    
    

}
}

@misc{fatkhullin,
      title={Can SGD Handle Heavy-Tailed Noise?}, 
      author={Ilyas Fatkhullin and Florian Hübler and Guanghui Lan},
      year={2025},
      eprint={2508.04860},
      archivePrefix={arXiv},
      primaryClass={math.OC},
      url={https://arxiv.org/abs/2508.04860}, 
}

@article{weiszfeld,
  title={Sur le point pour lequel la somme des distances de n points donn{\'e}s est minimum},
  author={Weiszfeld, Endre},
  journal={Tohoku Mathematical Journal, First Series},
  volume={43},
  pages={355--386},
  year={1937},
  publisher={Mathematical Institute, Tohoku University}
}
\end{document}